\documentclass[12pt]{amsart}
\usepackage[utf8]{inputenc}
\usepackage{amsmath,amssymb,amsthm,xcolor}
\usepackage[foot]{amsaddr}
\usepackage[margin=1in]{geometry}
\usepackage{hyperref}
\usepackage{graphicx}
\usepackage{tikz,tikz-cd}
\usetikzlibrary{matrix}
\usetikzlibrary{decorations.markings}
\usetikzlibrary{svg.path} 
\usetikzlibrary{arrows.meta}
\usepackage{subfig}
\usepackage{cleveref}

\newcommand{\rd}{\mathrm{d}}
\newcommand{\U}{\mathsf{U}}
\newcommand{\V}{\mathsf{V}}

\newtheorem{thm}{Theorem}
\newtheorem{cor}{Corollary}
\newtheorem{rem}{Remark}
\newtheorem{lma}{Lemma}

\title{An Analysis of Constraint-Relaxation in PDE-Based Inverse Problems}

\author{Tristan van Leeuwen$^{1,2}$}
\address{$^1$Centrum Wiskunde \& Informatica, Amsterdam, The Netherlands}
\address{$^2$Utrecht University, Utrecht, The Netherlands}

\author{Yunan Yang$^3$}
\address{$^3$Cornell University, Ithaca, NY, United States of America}

\email{t.van.leeuwen@cwi.nl}
\email{yunan.yang@cornell.edu}

\date{\today}

\begin{document}

\maketitle

\begin{abstract}
Inverse problems are ubiquitous in science and engineering. Many of these are naturally formulated as a PDE-constrained optimization problem. These non-linear, large-scale, constrained optimization problems know many challenges, of which the inherent non-linearity of the problem is an important one. In this paper we focus on a relaxed formulation of the PDE-constrained optimization problem and provide an in-depth analysis of it. Starting from an infinite-dimensional formulation of the inverse problem with discrete data, we propose a general framework for the analysis and discretisation of such problems. The relaxed formulation of the PDE-constrained optimization problem is shown to reduce to a weighted non-linear least-squares problem. The weight matrix turns out to be the Gram matrix of solutions of the PDE, and in some cases be estimated directly from the measurements.
The latter observation points to a potential way to unify recently proposed data-driven reduced-order models for inverse problems with PDE-constrained optimization. We provide a number of representative case studies and numerical examples to illustrate our findings.
\end{abstract}

\section{Introduction}

Inverse problems are ubiquitous in science and engineering and can often be expressed as a PDE-constrained optimization problem, where one aims to identify the parameters of a PDE give partial measurements of its solution. We built on the work presented in \cite{van2015penalty} where the PDE-constrained problem is posed as a joint parameter-state estimation problem which enforces the PDE-constraints only approximately. Effectively, this lifts the search space from the parameter alone to a much larger joint parameter-state space. It has been observed that this can mitigate the non-linearity of the optimization problem to some extent~\cite{van2015penalty, fang2020lift}. This relaxation approach is just one of many that aim to address the non-linearity and constitutes a line of research that is particularly prominent in exploration seismology and other wave-based inverse problems (see e.g., \cite{Symes2014,Warner2016,symes2020a,Symes2020} and references therein). The focus of this work, however, is the analysis of the relaxation described in \cite{van2015penalty}.

Concretely, the inverse problem is formulated as estimating coefficients $c$ from given data by solving
\begin{equation}\label{eq:intro}
\min_{c,u_1, \ldots, u_n} {\textstyle{\frac{1}{2}}}\sum_{i,j=1}^n |\mathcal{P}_iu_j - d_{ij}|^2 +  {\textstyle{\frac{\rho}{2}}}\|\mathcal{L}(c)u_j - \mathcal{P}_j\|_\mathsf{V}^2,
\end{equation}
where $\mathcal{L}(c):\mathsf{U}\rightarrow\mathsf{V} = \mathsf{U}^*$ denotes the partial differential operator with coefficient $c$, $\mathcal{P}_i\in \V$ denotes the linear sampling operator which also acts as the source terms, and $d_{ij} = \mathcal{P}_i\check{u}_j$ represents the observed data corresponding to the underlying true state $\check{u}_j = \mathcal{L}(\check{c})^{-1}\mathcal{P}_j$ for the true parameter $\check{c}$, and $\rho > 0$ is a penalty parameter. We denote the adjoints as $\mathcal{L}^* : \mathsf{U}\rightarrow\mathsf{V}$. Furthermore, both $\mathcal{L}$ and $\mathcal{L}^*$ have a well-defined inverse, denoted by $\mathcal{L}^{-1}, \mathcal{L}^{-*}$. The Riesz map is denoted as $\mathcal{R}:\mathsf{V}\rightarrow\mathsf{U}$.

\subsection{Contributions}
It has been shown empirically that this reformulation can improve the optimization landscape, making the iterative process less sensitive to initialization, and hence leading to more robust inversion results~\cite{van2015penalty}. \emph{We innovate upon this previous work in the following three directions:}

\emph{First}, we treat the optimization problem~\eqref{eq:intro} in the continuous setting and thereby provide a common framework for analysis and numerical implementation of a wide range of inverse problems based on the weak form of the PDE. Considering a finite number of measurements, we present a Representer Theorem (see \Cref{thm:Representer}) which shows that the estimated state lives in a finite-dimensional subspace of $\mathsf{U}$. This leads to a re-formulation of the traditional reduced approach for PDE-constrained optimization with an objective function equipped with a \textit{parameter-dependent residual weight}. In particular, it simplifies to
\[
\min_c  J(c):={\textstyle{\frac{1}{2}}}\sum_{i=1}^n\|\mathbf{e}_i(c)\|_{(I + \rho^{-1} G(c))^{-1}}^2,
\]
where $\mathbf{e}_i(c) \in \mathbb{R}^n$ denotes the $i^\text{th}$ data-residual whose $j$-th element is given by $e_{ij}(c) = \mathcal{P}_i\mathcal{L}(c)^{-1}\mathcal{P}_j- d_{ij}$, and $G(c)\in \mathbb{R}^{n\times n}$ is a positive semi-definite matrix with elements $G_{ij}(c) = \mathcal{P}_i\left(\mathcal{L}^*(c)\mathcal{R}\mathcal{L}(c)\right)^{-1}\mathcal{P}_j$, $1\leq i,j\leq n$. %

\emph{Second}, based on the Representer Theorem, we analyze the two limiting cases, and provide an interpretation of the finite-dimensional residuals $\mathbf{e}_i\in\mathbb{R}^n$ in terms of the underlying function spaces $\mathsf{U}$ and $\mathsf{V}$. In particular, we show that the relaxed formulation~\eqref{eq:intro} embodies a rich interplay between
\begin{itemize}
    \item the finite-dimensional \emph{data-residual} $\mathbf{e}_i(c)\in \mathbb{R}^n$,
    \item the \emph{solution-residual} $\mathcal{L}(c)^{-1}\mathcal{P}_i - \check{u}_i \in \mathsf{U}$,
    \item the \emph{PDE-residual}  $\mathcal{P}_i - \mathcal{L}(c)\check{u}_i \in \mathsf{V}$.
\end{itemize}
Under some simplifying assumptions, we show the following limiting behaviour of $J$:
\begin{itemize}
\item[$\rho\rightarrow \infty$:] $J(c)$ measures the norm of the \emph{solution-residual} projected on to a finite-dimensional subspace $\mathsf{P}_n = \text{span}\{\mathcal{R}\mathcal{P}_i\}_{i=1}^n \subset \mathsf{U}$.
 
\item[$\rho\rightarrow 0$:\,\,\,] $J(c)$ measures the norm of the \emph{PDE-residual}, projected on to a finite-dimensional subspace $\mathsf{W}_n = \text{span}\{\mathcal{R}^{-1}\mathcal{L}(c)^{-*}\mathcal{P}_i\}_{i=1}^n \subset \mathsf{V}$.
\end{itemize}

For a finite $\rho > 0$, the objective function interpolates between these two limiting situations. This is schematically depicted in Figure \ref{fig:one}. One might argue that it is more natural to use the PDE-residual because it can typically be parameterised to be affine in the coefficient $c$ for many common inverse problems, such as the Calder\'on problem~\cite{uhlmann2009electrical}, full-waveform inversion~\cite{van2015penalty} and inverse scattering~\cite{cakoni2005qualitative}. Precisely, in the $\rho\rightarrow 0$ case, under certain conditions, we can prove that the optimization problem is convex with respect to $c$ (see Corollary~\ref{thm:quadratic}), while the original objective function (when $\rho \rightarrow \infty$) is known to be highly non-convex, for all the examples mentioned above. 
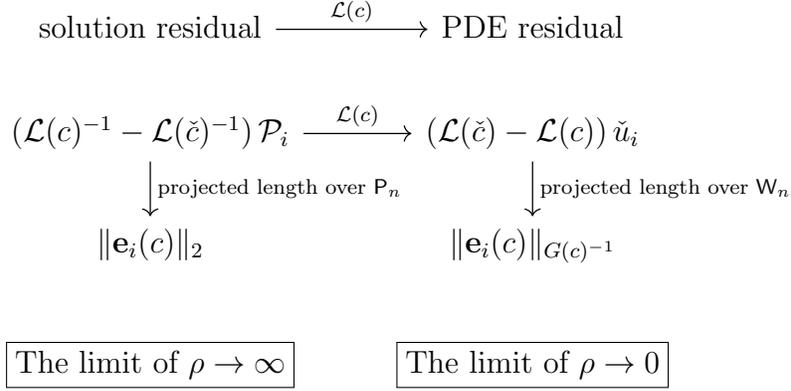
\begin{figure}
\begin{tikzcd} 
\text{solution residual}\arrow[r,"\mathcal{L}(c)"] &  \text{PDE residual}  \\
\left(\mathcal{L}(c)^{-1} - \mathcal{L}(\check{c})^{-1}\right)\mathcal{P}_i\arrow[d,"\text{projected length over $\mathsf{P}_n$}"]\arrow[r,"\mathcal{L}(c)"]&\left(\mathcal{L}(\check{c})-\mathcal{L}(c)\right)\check{u}_i\arrow[d,"\text{projected length over $\mathsf{W}_n$}"]\\ 
\|\mathbf{e}_i(c)\|_{2}&\|\mathbf{e}_i(c)\|_{G(c)^{-1}}\\
\boxed{\text{The limit of } \rho \rightarrow \infty}  & \boxed{\text{The limit of } \rho \rightarrow 0}
\end{tikzcd}
\caption{\label{fig:one}
The two limiting cases of the residual between simulated ($\mathcal{P}_i\mathcal{L}(c)^{-1}\mathcal{P}_j$) and measured ($d_{ij}=\mathcal{P}_i\mathcal{L}(\check{c})^{-1}\mathcal{P}_j$) data. When $\rho\rightarrow\infty$, we measure the $\ell^2$ norm of the \emph{data-residual}  ($e_{ij}(c) = \mathcal{P}_i\left(\mathcal{L}(c)^{-1} - \mathcal{L}(\check{c})^{-1}\right)\mathcal{P}_j$), which corresponds to the projected length of the \emph{solution-residual} $\left(\mathcal{L}(c)^{-1}- \mathcal{L}(\check{c})^{-1}\right)\mathcal{P}_i$ on $\mathsf{P}_n$. When $\rho\rightarrow 0$, we measure the \emph{$G(c)^{-1}$-weighted data residual}, which corresponds to the projected length of the \emph{PDE-residual} $\left(\mathcal{L}(c)-\mathcal{L}(\check{c})\right)\check{u}_i$ on  $\mathsf{W}_n $.}
\end{figure}

As a \emph{third contribution}, we show through a series of case studies how the matrix $G(\check{c})$ corresponding to the true coefficient can be estimated from the measured data $d_{ij}$ directly if the underlying space $\mathsf{U}$ is chosen appropriately. It turns out that this construction is closely related to recent progress on data-driven reduced order models pioneered by \cite{Borcea2018,Borcea2020,borcea2022waveform}.
Thus, this work goes towards bridging the gap between PDE-constrained optimization and data-driven reduced-order models and opens up an avenue for efficient implementation of the relaxed formulation.

\subsection{Outline}
The remainder of this paper is organized as follows: \Cref{sec:theory} introduces the proposed approach, detailing its formulation and integration with PDE-constraints. In~\Cref{sec:cases}, we present a series of case studies in which we apply the framework to a number of benchmark problems. Finally, \Cref{conclusions} concludes the paper with a summary of the findings and discussions on future research directions.

\section{Theory}\label{sec:theory}
In this section we present our main results. First, we pose the inverse problem as unconstrained optimization over an extended search space consisting of both the parameter $c$ and the states $\{u_i\}_{i=1}^n$, followed by an equivalent reformulation in terms of the parameter $c$ and auxiliary source terms $\{q_i\}_{i=1}^n$. Then, we present a Representer Theorem which expresses the solution of the optimization problem over states or sources as a finite linear combination of basis functions. Using this representation, we then reduce the PDE-constrained optimization problem to an unconstrained optimization problem over the parameters only. In particular, we show that this formulation is equivalent to a conventional reduced formulation with a \textit{coefficient-dependent} weight on the residuals.

\subsection{Preliminaries}
For the remainder of the paper we consider the weak form of the PDE
\begin{equation}\label{eq:weak}
\mathcal{A}_c(u, \phi) = \mathcal{P}(\phi), \quad \forall\phi\in \mathsf{U},
\end{equation}
where $\mathcal{A}_c:\mathsf{U}\times \mathsf{U} \rightarrow \mathbb{C}$ is a sesquilinear form (i.e., linear in the first argument and anti-linear in the second argument) with parameters $c \in \mathsf{C}$ and $\mathcal{P} \in \overline{\mathsf{U}^*}$ is an \emph{antilinear} operator. These slightly unconventional definitions are chosen to treat the case where $\mathsf{U}$ is a complex Hilbert space. We furthermore introduce the \emph{linear} Riesz map $\mathcal{R}:\overline{\mathsf{U}^*}\rightarrow \mathsf{U}$ which allows us to represent any $\mathcal{P}\in \overline{\mathsf{U}^*}$ by a corresponding $p\in\mathsf{U}$, i.e., $p = \mathcal{R}\mathcal{P}$, such that 
\[
\mathcal{P}(u) = \langle \mathcal{R}\mathcal{P}, u\rangle_{\mathsf{U}},\quad \langle p, u\rangle_{\mathsf{U}} = \left(\mathcal{R}^{-1}p\right)(u),
\]
where $\langle \cdot, \cdot\rangle_\mathsf{U}$ denotes the inner product in $\mathsf{U}$, which is also defined to be linear in the first argument and anti-linear in the second. Note that this definition of the Riesz map for complex-valued replaces the definition presented in the Introduction.

Throughout, we assume that the weak formulation is well-posed uniformly for all $c \in \mathsf{C}$ and that the Riesz representations $\{p_i\}_{i=1}^n$ of the functionals $\{\mathcal{P}_i\}_{i=1}^n$ are linearly independent. Given solutions $\{u_i\}_{i=1}^n$ corresponding to source terms $\{\mathcal{P}_i\}_{i=1}^n$, we consider measurements
\[
d_{ij} = \mathcal{P}_i(u_j),\quad i,j = 1,\ldots, n\,.
\]
The goal is to estimate the parameters $c$ from these measurements. The classical approach to doing so is by solving a PDE-constrained optimisation problem:
\begin{equation}\label{eq:FWI}
 \min_{c\in\mathsf{C}}  {\textstyle\frac{1}{2}} \sum_{i,j=1}^{n}   \big|\mathcal{P}_i(u_j) - d_{ij}\big|^2\,\,\,\,\text{s.t.}\,\,\,\,\mathcal{A}_c(u_j,\phi) = \mathcal{P}_j(\phi),\,\,\forall\phi\in \mathsf{U}.
 \end{equation}
{Even though this notation suggests that we assume full-array data (i.e., each source also acts as a receiver) we can easily modify this by including a binary weight $s_{ij}$ in the double sum to select only those source-receiver combinations that are actually measured. For ease of notation we leave it out for most derivations and will make specific remarks for those results that generalize to the case of incomplete data.}

This optimization problem often suffers from very non-convex misfit landscapes that are challenging for gradient-based optimization algorithms. Some of the difficulties come from the ``hard'' equality constraint $\mathcal{A}_c(u_j,\phi) = \mathcal{P}_j(\phi)$, $\forall \phi\in \U$. Motivated by the challenges, we instead consider a relaxation of this problem by using a ``soft'' constraint, as proposed by~\cite{van2015penalty}.
  
\subsection{Constraint-relaxation}
An alternative to the constrained optimization problem~\eqref{eq:FWI} is the following:
\begin{equation}\label{eq:eFWI}
\min_{c\in\mathsf{C},q\in\mathsf{U}^n}   {\textstyle\frac{1}{2}} \sum_{i,j=1}^{n}   \big|\mathcal{P}_i(u_j) - d_{ij}\big|^2 +   {\textstyle\frac{\rho}{2}} \sum_{j=1}^n  \|q_j\|_{\mathsf{U}}^2\,\,\,\,\text{s.t.}\,\,\mathcal{A}_c(u_j,\phi) = \left(\mathcal{P}_j + \mathcal{R}^{-1}q_j\right)(\phi),\,\,\forall\phi\in \mathsf{U},
 \end{equation}
where $q = (q_1, q_2, \ldots, q_n)$ represent auxiliary source terms, and $\rho > 0$ is a trade-off parameter. Rather than enforcing that each $u_j$ has to be exactly the PDE solution for a fixed parameter $c$ with the source term $\mathcal{P}_j$, we allow the source term to ``deviate'' from $\mathcal{P}_j$ by a quantity determined by $q_j$, and then penalize the size of $q_j$ through $\frac{\rho}{2} \|q_j\|_\U$. 

The introduced extra degree of freedom, i.e., an enlarged model capacity,  can help match the data more easily. It is also a relaxation from the original ``hard'' PDE constraint as~\eqref{eq:eFWI} will reduce to the traditional approach~\eqref{eq:FWI} in the limit of $\rho\rightarrow \infty$. In this sense, this formulation fits in the context of model extension, pioneered by \cite{Symes2014,Warner2016,Symes2020,symes2020a}. It has been shown that enlarging the search space can make the optimization problem have a better optimization landscape, and thus make it less sensitive to initialisation of the parameter $c$ when using gradient-based optimization algorithms. However, it is computationally not feasible to explicitly optimise over the enlarged space $\mathsf{C}\times \mathsf{U}^n$. To address this, we show in~\Cref{thm2} that this extended formulation is equivalent to minimizing a weighted residual over $\mathsf{C}$ only through the application of the Representer Theorem (see~\Cref{thm:Representer}).

Before moving forward to the theory, we first equivalently reformulate~\eqref{eq:eFWI} as
\begin{eqnarray}\label{eq:eFWI2}
\min_{c\in\mathsf{C},q\in\mathsf{U}^n}  {\textstyle\frac{1}{2}} \sum_{i,j=1}^{n}   \big|\mathcal{P}_i(v_j) - e_{ij}\big|^2 +   {\textstyle\frac{\rho}{2}} \sum_{j=1}^n  \|q_j\|_{\mathsf{U}}^2,\\
 \begin{split}
\text{s.t.}\,\,\mathcal{A}_c(u_j,\phi) &= \mathcal{P}_j(\phi)\,&\forall\phi\in \mathsf{U},\nonumber\\
\mathcal{A}_c(v_j,\phi) &= \mathcal{R}^{-1}q_j(\phi)\,&\forall\phi\in \mathsf{U},\nonumber
 \end{split}
 \end{eqnarray}
where the term 
\begin{equation}\label{eq:new_error}
    e_{ij} = d_{ij} - \mathcal{P}_i(u_j)
\end{equation}
depends on $c$ implicitly through $u_j$. We will use~\eqref{eq:eFWI2} for the following derivations.

\subsection{A representer theorem}
When the parameter $c$ is fixed, both~\eqref{eq:eFWI} and~\eqref{eq:eFWI2} are reduced to an inner problem, which is quadratic with respect to $q$. Next, we show that its solution has a finite-dimensional representation. To facilitate the discussion, we split the inner problems for $\{q_j\}_{j=1}^n$ in ~\eqref{eq:eFWI2} in $n$ independent quadratic subproblems and introduce
\begin{equation}\label{eq:eFWI_inner}
J_j(q_j) =  {\textstyle\frac{1}{2}}\sum_{i=1}^{n}   \big|\mathcal{P}_i(v_j) - e_{ij}\big|^2 +   {\textstyle\frac{\rho}{2}}   \|q_j\|_{\mathsf{U}}^2,
\end{equation}
where $v_j$ depends on $q_j\in \mathsf{U}$ through the weak form of the PDE:
\[
\mathcal{A}_c(v_j,\phi) = \langle q_j,\phi \rangle_\mathsf{U}\quad\forall\phi\in \mathsf{U}.
\]
Note that, since $c$ is fixed, $e_{ij}$ is fixed as defined above. 

\begin{thm}[A Representer Theorem]
\label{thm:Representer}
The functional defined in \eqref{eq:eFWI_inner} admits minimizers of the form
\begin{equation}\label{eq:representation}
q_j = \sum_{k=1}^n \overline{\alpha_{jk}}w_k,
\end{equation}
where $w_k \in \mathsf{U}$ satisfies an adjoint equation
\begin{equation}\label{eq:weakadjoint}
\overline{\mathcal{A}_c(\phi, w_k)} = \mathcal{P}_k(\phi),
\end{equation}
and the complex-valued coefficients $\{\alpha_{jk}\}_{k=1}^n$ are solutions to the following linear system
\[
\sum_{k} g_{ik}\alpha_{jk} + \rho \alpha_{ji} = e_{ij},\quad i=1, \ldots, n,
\]
with $g_{ik} = \langle w_i, w_k\rangle_\mathsf{U}$.
\end{thm}

\begin{proof}
Noting that
$$\mathcal{P}_i({v}_j) = \overline{\mathcal{A}_c(v_j, w_i)} = \overline{ \langle q_j, {w}_i \rangle_{\mathsf{U}}}  = \langle {w}_i, q_j \rangle_{\mathsf{U}},$$
we can rewrite~\eqref{eq:eFWI_inner} as
$$J_j(q_j)={\textstyle\frac{1}{2}}\sum_{i=1}^n |\langle w_i,{q_j}\rangle_{\mathsf{U}} - e_{ij}|^2 + {\textstyle{ {\textstyle\frac{\rho}{2}}}}\|q_j\|_{\mathsf{U}}^2.$$
We can now directly apply Lemma \ref{lemma:optimality} (see Appendix) to find the desired result. 
\end{proof}
{%
\begin{rem}
As we apply the representer theorem separately to each source index, $j$, the statement in Theorem 1 extends to the case of incomplete array data by only including the relevant receivers for each source.
\end{rem}
}
\subsection{A reduced formulation}
With the above-stated result, we can now derive a reduced formulation of~\eqref{eq:eFWI2} that depends on the parameter $c$ only.
\begin{cor}
[Variable metric reduced formulation]
\label{thm2}
The relaxed problem \eqref{eq:eFWI2} in $(c, q) \in \mathsf{C}\times \mathsf{U}^n$ can be formulated equivalently in reduced form as
\[
\min_{c\in\mathsf{C}} J(c),
\]
with
\begin{equation}\label{eq:new_obj}
J(c) =  {\textstyle\frac{1}{2}}\sum_{j=1}^n\|\mathbf{e}_j(c)\|_{(I+\rho^{-1}G(c))^{-1}}^2 =   {\textstyle\frac{1}{2}}\text{trace}\left[E(c)^*(I+\rho^{-1}G(c))^{-1}E(c)\right],
\end{equation}
with
${\mathbf{e}}_j(c) = [e_{1j}(c),\ldots,e_{nj}(c)]^\top$, and the matrix $G \in \mathbb{C}^{n\times n}$ as defined in Theorem \ref{thm:Representer}.
\end{cor}
\begin{proof}
Using the result from Theorem~\ref{thm:Representer}, we have that, at the optimizer,
$$\langle w_i, q_j\rangle_\mathsf{U} = \sum_{k=1}
^n g_{ik}\alpha_{jk},$$
$$\|q_j\|_{\mathsf{U}}^2 = \sum_{k=1}^n \sum_{l=1}^n\overline{\alpha_{jk}}\alpha_{jl}g_{k,l}.$$
Then
$$J_j(q_j)=  {\textstyle\frac{1}{2}}\|G\boldsymbol{\alpha}_j - \mathbf{e}_j\|_2^2 +  {\textstyle\frac{\rho}{2}}\boldsymbol{\alpha}_j^* G\boldsymbol{\alpha}_j.$$
Using that $(G + \rho I)\boldsymbol{\alpha}_j = \mathbf{e}_j$ and the Woodbury matrix identity, we immediately find
$$J_j(q_j) =  {\textstyle\frac{1}{2}}\mathbf{e}_j^*(I + \rho^{-1}G)^{-1}\mathbf{e}_j,$$
after simplification, which yields~\eqref{eq:new_obj}.
\end{proof}
{\begin{rem}
In the case of incomplete array data, the Gram matrix would generally depend on $j$, containing only the inner-products of the adjoint solutions for the corresponding receivers.
\end{rem}
}
\subsection{Limiting cases}
Here, we analyse the limiting behaviour of the objective defined in~\eqref{eq:new_obj} for $\rho\rightarrow 0$ and $\rho \rightarrow \infty$, which sheds light on the interpolative nature for a finite $\rho >0$. 
\begin{lma}
For $\rho \rightarrow \infty$ we can express the objective as
\[
J(c) =  {\textstyle\frac{1}{2}}\text{trace}\left[E(c)^*E(c)\right] + \mathcal{O}(\rho^{-1}).
\]
\end{lma}
\begin{proof}
This follows directly from \eqref{eq:new_obj} by expressing $(I + \rho^{-1}G)^{-1} = \sum_{k=0}^\infty (-\rho)^{-k}G^k$ which is valid for $\rho > \|G\|$, i.e., $\rho$ is larger than the largest eigenvalue of $G$.
\end{proof}
\begin{lma}
\label{lma:rho_zero}
We can write the objective for $\rho \rightarrow 0$ as
\[
J(c) = {\textstyle\frac{\rho}{2}}\text{trace}\left[E(c)^*G(c)^{-1}E(c)\right] + \mathcal{O}(\rho^2).
\]
\end{lma}
\begin{proof}
Since we assumed that $\{p_i\}_{i=1}^n$ are linearly independent, so are $\{w_i\}_{i=1}^n$ and $G$ is invertible. We can then express $(I + \rho^{-1}G)^{-1}=\rho G^{-1}(I + \rho G^{-1})^{-1} = \rho G^{-1}\sum_{k=0}^{\infty} (-\rho)^k G^{-k}$, which is valid for $ \|G^{-1}\| < 1/\rho$, i.e., $\rho$ should be smaller than the smallest eigenvalue of the matrix $G$. The desired result follows immediately.
\end{proof}
Based on these limiting cases, we define the following
\begin{eqnarray}
J_\infty(c) &=& \lim_{\rho\rightarrow \infty} J(c) ={\textstyle\frac{1}{2}}\text{trace}\left[E(c)^*E(c)\right]    \label{eq:Jinfty}\\
J_0(c) &=& \lim_{\rho\rightarrow 0} \rho^{-1}J(c) =  {\textstyle\frac{1}{2}}\text{trace}\left[E(c)^*G(c)^{-1}E(c)\right].\label{eq:J0}
\end{eqnarray}
{\begin{rem}
The asymptotic analysis also holds for the case of incomplete array data, albeit that the resulting expressions are slightly less compact.
\end{rem}
}
\subsection{Noiseless data}
We can furthermore interpret these limiting cases in terms of projections of certain infinite-dimensional residuals on finite-dimensional subspaces. We assume noisless data $d_{ij} = \mathcal{P}_i(\check{u}_j)$ with $\check{u}_j$ the (weak) solution corresponding to $\mathcal{P}_j$ for the true coeffienct $\check{c}$, $1\leq i,j \leq n$.

\begin{thm}\label{thm:sol_proj}
When the Riesz representers $\{p_i\}_{i=1}^n$ are orthonormal and the measurements are noiseless, the functional $J_\infty$ defined in \eqref{eq:Jinfty} can be equivalently expressed in terms of an orthogonal projection of the solution-residual $u_i(c) - \check{u}_i \in \mathsf{U}$ on $\mathsf{P}_n = \text{span}\{p_i\}_{i=1}^n$ as
\[
J_{\infty}(c) = {\textstyle\frac{1}{2}}\sum_{j=1}^n\|\Pi_{\mathsf{P}_n}(u_j(c) - \check{u}_j)\|_{\mathsf{U}}^2,
\]
where $\Pi_{\mathsf{P}_n}$ is the orthogonal projection on $\mathsf{P}_n$, $u_j(c)$ is the weak solution corresponding to coefficient $c$, and $\check{u}_j$ is the true state corresponding to the true coefficient $\check{c}$.
\end{thm}
\begin{proof}
Since we assumed $\{p_i\}_{i=1}^n$ to be orthonormal, the projection of the solution residual $u_j(c) - \check{u}_j$ over $\mathsf{P}_n $ is given by 
\[
\Pi_{\mathsf{P}_n}(u_j(c) - \check{u}_j) = \sum_{i=1}^n \langle u_j(c) - \check{u}_j,p_i\rangle_{\mathsf{U}}\,p_i.
\]
Taking the norm we obtain
\[
{\textstyle\frac{1}{2}}\sum_{j=1}^n\|\Pi_{\mathsf{P}_n}(u_j(c) - \check{u}_j)\|_{\mathsf{U}}^2 = \sum_{j=1}^n \left\|\sum_{i=1}^n \overline{e_{ij}(c)}p_i\right\|_{\mathsf{U}}^2 = \sum_{j,i,i'=1}^n \overline{e_{ij}(c)}\langle p_i,p_{i'}\rangle_\mathsf{U}\, {e_{i'j}(c)},
\]
with $e_{ij}(c) = d_{ij} - \langle p_i, u_j\rangle_{\mathsf{U}}$ as defined before (see Equation~\eqref{eq:new_error}).
Using orthonormality of $\{p_i\}_{i=1}^n$ again we get the desired result.
\end{proof}
{In Theorem~\ref{thm:sol_proj} below, we assume that $\{p_i\}_{i=1}^n$ forms an orthonormal basis. If this is not the case, $J_\infty$ still measures the $\|\cdot\|_{\mathsf{U}}$ norm of the projected solution residual, but it represents a weighted rather than an orthogonal projection.}

\begin{thm}\label{thm:PDE_proj}
For noiseless measurements, the functional $J_0(c)$ can be equivalently expressed in terms of an orthogonal projection of the PDE-residual $\mathcal{E}_j(c) = \mathcal{A}_c(\check{u}_j,\cdot) -  \mathcal{A}_{\check{c}}(\check{u}_j,\cdot) \in \mathsf{V}$, $j=1,\ldots,n$, on $\mathsf{W}_n = \text{span}\{w_j\}_{j=1}^n$ as
\[
J_0(c) = {\textstyle\frac{1}{2}}\sum_{j=1}^n \|\Pi_{\mathsf{W}_n}\mathcal{R}\mathcal{E}_j(c)\|_{\mathsf{U}}^2,
\]
with $\check{u}_j$ denoting the true state corresponding to the true coefficient $\check{c}$, and $w_j$ the adjoint state defined in~\eqref{eq:weakadjoint} with the current coefficient $c$.
\end{thm}
\begin{proof}
Consider the Riesz presentation of the PDE-residual $s_j(c) := \mathcal{R}\mathcal{E}_j(c) \in \U $, for any $j= 1,\ldots, n$. Its orthogonal projection on $\mathsf{W}_n = \text{span}\{w_i\}_{i=1}^n \subset \U$ is given by
\[
\Pi_{\mathsf{W}_n} s_j(c) = \sum_{i=1}^n \overline{\alpha_{ij}} w_i\,,
\]
where $\boldsymbol{\alpha}_{j} = G^{-1}\overline{\mathbf{r}_j}$, $\boldsymbol{\alpha}_{j} = [{\alpha}_{1j},\ldots, {\alpha}_{nj}]^\top$, $\mathbf{r}_{j} = [{r}_{1j},\ldots, {r}_{nj}]^\top$  with $r_{ij} = \langle s_j(c),w_i\rangle_{\mathsf{U}} =  \langle \mathcal{R}\mathcal{E}_j(c) , w_i\rangle_{\U} = \mathcal{E}_j(c) w_i$ and the matrix $G$ with entries $g_{ij}=\langle w_i, w_j\rangle_\mathsf{U}$. First note that, since $\mathcal{A}_{\check{c}}(\check{u}_j,w_i)  = \mathcal{A}_{c}({u}_j,w_i)  = \mathcal{P}_j(w_i)$,
\[
r_{ij} = \mathcal{A}_c(\check{u}_j,w_i) -  \mathcal{A}_{\check{c}}(\check{u}_j,w_i) = \mathcal{A}_c(\check{u}_j,w_i) -  \mathcal{A}_{c}({u}_j,w_i) = \overline{\langle p_i,\check{u}_j\rangle} - \overline{\langle p_i,u_j\rangle} = \overline{e_{ij}},
\]
with $e_{ij} =d_{ij} -  \mathcal{P}_i(u_j) $ as defined in~\eqref{eq:new_error}. Taking the norm it follows that 
\[
\|\Pi_{\mathsf{W}_n} s_j(c)\|_{\mathsf{U}}^2 = \boldsymbol{\alpha}_j^*G\boldsymbol{\alpha}_j = \|\mathbf{e}_j\|^2_{G^{-1}}\,,
\]
where $\mathbf{e}_j = [e_{1j}, \ldots, e_{ij},\ldots, e_{nj}]^\top$.
\end{proof}
\begin{rem}
The results in  Theorems~\ref{thm:sol_proj} and~\ref{thm:PDE_proj} are the analogue of the relations sketched in Figure \ref{fig:one}.
\end{rem}
{\begin{rem}Theorems~\ref{thm:sol_proj} and~\ref{thm:PDE_proj} can be extended to incomplete array data by allowing a different orthogonal projector for each source index $j$.
\end{rem}}
\subsection{Towards a convex formulation}\label{sec:convex}
So far, the Riesz representations $\{p_i\}_{i=1}^n$ followed from the definition of $\{\mathcal{P}_i\}_{i=1}^n$ and $\mathsf{U}$. Alternatively, we can start from $\{p_i\}_{i=1}^n$ and consider a \emph{finite-dimensional} space $\mathsf{U} = \mathsf{P}_n = \text{span}\{p_i\}_{i=1}^n$. That is, the underlying solution space are precisely spanned by the Riesz representations of the source/measurement operators. In this setup, the objective $J_0$ then simplifies significantly, as is shown in the following Corollary.
\begin{cor}
\label{thm:quadratic}
Suppose $\mathsf{U} = \text{span}\{p_i\}_{i=1}^n$, where $\{p_i\}_{i=1}^n$ are real-valued, sufficiently regular and linearly independent functions. Then it holds thats
\begin{equation}
J_{0}(c) =  {\textstyle\frac{1}{2}}\text{trace}\left[\left(M - A(c)M^{-1}D\right)^*M^{-1}\left(M - A(c)M^{-1}D\right)\right],
\end{equation}
with $M_{ij} = \langle p_i, p_j \rangle_{\mathsf{U}}$ and $A(c)_{ij} = \overline{\mathcal{A}_c(p_j, p_i)}$.
\end{cor}
\begin{proof}
Following the Galerkin approach, for each $j$, we let
\[
u_j = \sum_{k=1}^n U_{kj}\,p_k, 
\]
where the coefficients $U_{kj} \in \mathbb{C}$ are solved from
\[
\sum_{k=1}^n \mathcal{A}_c(p_k, p_i)U_{kj} = \langle p_j, p_i\rangle_\mathsf{U} \quad \text{for $i=1, 2, \ldots, n$}.
\]
The $n$-by-$n$ matrix $U$ has the $kj$-entry $U_{kj}$, $1\leq k,j \leq n$.

Note that this assumes that the functions $\{p_i\}_{i=1}^n$ are sufficiently regular to be considered as the test functions in $\mathcal{A}_c(p_k, p_i)$.
Similarly, assume the adjoint equation solution corresponding to the $j$-th source term is $\sum_{k=1}^n{W_{kj}} p_k $, $1\leq j \leq n$. Then coefficients  for the adjoint wavefields, $W_{kj}$, $1\leq k,j\leq n$, which gives rise to the matrix $W$, are solved from 
\[
\sum_{k=1}^n \overline{\mathcal{A}_c(p_i, p_k)}W_{kj} = \langle p_j, p_i\rangle_\mathsf{U} \quad \text{for $i=1, 2, \ldots, n$}.
\]
Then the data residual matrix $E(c)$ and the Gram matrix $G(c)$ have their $ij$-th entries to be
\[
e_{ij} = d_{ij} - \left\langle p_i, \sum_{k=1}^n U_{kj}p_k\right\rangle_{\mathsf{U}} = d_{ij} - \sum_{k=1}^n \langle p_i, p_k \rangle_\mathsf{U}\overline{U_{kj}},
\]
and
\[
g_{ij} = \left\langle \sum_{k=1}^n W_{ki}p_k\,,\, \sum_{\ell=1}^n W_{\ell j}p_\ell \right\rangle_\mathsf{U} = \sum_{k,\ell} W_{ki} \overline{W_{\ell j}} \langle p_k, p_\ell \rangle_\mathsf{U}.
\]
We introduce two more  $n$-by-$n$ matrices $M$ and $A$ 
 with entries $m_{ij} = \langle p_i, p_j \rangle_{\mathsf{U}}$, $a(c)_{ij} = \overline{\mathcal{A}_c(p_j, p_i)}$. Note that, by our assumptions, both matrices are invertible. Moreover, $M$ is Hermitian, i.e., $M = M^*$,  where $M^*$ denotes the Hermitian transpose of $M$. We now find
\[
E(c) = D - M\overline{U(c)} =  D - MA(c)^{-1}M,
\]
\[
G(c) = \overline{W(c)}^* M \overline{W(c)} = MA(c)^{-1}MA(c)^{-*}M\,.
\]
Substituting this in the result of Lemma~\ref{lma:rho_zero} we get the desired result.
\end{proof}

\begin{rem}\label{rem:quadratic}
If $\mathcal{A}_c$ is affine in $c$,  and the inner produce $\langle \cdot, \cdot \rangle_{\mathsf{U}}$ does not depend on $c$,  then $J_0(c)$ is \textit{quadratic} in $c$, which is ideal for many gradient-based optimization algorithms. However, this result does \emph{not} imply that $c$ is uniquely recoverable from the measurements, as this additionally would require $A(c_1) = A(c_2) \Rightarrow c_1 = c_2 \,, \forall \, c_1, c_2 \in \mathsf{C}$.  Moreover, we do not have any guarantees that the minimizer will coincide with the true coefficient as this choice for $\mathsf{U}$ may not yield a reasonable approximation of the underlying physics described by the PDE.
\end{rem}

\begin{rem}\label{rem:direct}
This result suggests a direct method for solving the inverse problem by setting up a system of $n^2$ (linear) equations of the form
\[
A(c) = MD^{-1}M,
\]
for a suitable parameterisation of $c$. If $\mathcal{A}_c$ is affine in $c$,  we only need to  solve a linear system with respect to $c$.
\end{rem}

\begin{rem}\label{rem:linear}
If $\mathcal{A}_c$  and $\langle \cdot, \cdot \rangle_{\mathsf{U}}$ are both affine in $c$,  in particular,  if $\langle p_i,  p_j\rangle_{\mathsf{U}} = \mathcal{A}_c (p_i, p_j)$ which implies that $M = A(c)$,  then  $J_0(c)$ can still be convex under suitable conditions.  %
\end{rem}

\section{Case studies}\label{sec:cases}
In this section we present a few case studies. The choice of $\mathsf{U}$ plays a central role. The two important aspects that we treat for each are \emph{i)} well-posedness of the weak form for the chosen space, and \emph{ii)} how to compute the Grammian matrix $G$ from the given measurements.

\subsection{1D Elliptic equation}
We consider a 1D elliptic PDE 
\[
-\left( c(x) u'(x)\right)' =  f(x), \quad x\in \Omega,
\]
with the boundary condition $u|_{\partial \Omega} = 0$. The corresponding forward (PDE) operator is denoted by 
\[\mathcal{L}(c) : H^{1}_0(\Omega) \longrightarrow H^{-1}(\Omega)\,.
\]
Here, $\Omega \subset \mathbb{R}$ is a bounded domain for the solution $u$, and $u|_{\partial \Omega}$ denotes its boundary. The short-hand notation for the PDE is $\mathcal{L}(c) u = f \in H^{-1}(\Omega)$. 

We further assume that the coefficient $c(x)$ is bounded both from above and below, i.e., $0 < a \leq c(x) \leq b < \infty$.  Therefore,  the PDE is uniformly elliptic.

\subsubsection{Weak formulation}
We consider real-valued functions, so $\mathcal{A}$ is bilinear and $\mathcal{P}$ is a linear operator:
$$\mathcal{A}(u,\phi) = \int_\Omega \phi(x) \mathcal{L}(c) u(x) \mathrm{d}x\,, $$
$$\mathcal{P}(u) = \int_{\Omega} f(x)u(x) \mathrm{d}x\,. $$
We let $\mathsf{U} = H_0^1(\Omega)$ with a weighted $H^{1}(\Omega)$-inner product:
\begin{equation}\label{eq:1D_E_U}
    \langle u, v\rangle_{\mathsf{U}} = \int_\Omega v(x) \mathcal{L}(c) u(x) \mathrm{d}x =  \int_{\Omega} c(x)u'(x)v'(x) \mathrm{d}x.
\end{equation}
The operator $\mathcal{P}$ is well-defined with $f \in H^{-1}(\Omega)$. We can view the Riesz map as $\mathcal{R}:H^{-1}(\Omega)\rightarrow H^{1}_0(\Omega)$, defined through solving the PDE, $ -  \left(c(x) u(x)'\right)' = f(x)$, with homogeneous Dirichlet boundary conditions. 

Let $\mathsf{U}$ be the space $H^{1}_0(\Omega)$ functions but equipped with the weighted inner product  given in~\eqref{eq:1D_E_U}, which induces the norm $\|\cdot \|_{\mathsf{U}} = \sqrt{\langle u, u \rangle_\mathsf{U}}$. Note that this is an equivalent norm to the standard $H^1$ norm since $  a \|\cdot \|_{H^1} \leq \|\cdot \|_{\mathsf{U}} \leq a^{-1} \|\cdot \|_{H^1}$ for some constant $a>0$ given by the Poincar\'e inequality.

\subsubsection{Data-driven kernel}
The Green's function for $\mathcal{L}(c)$, denoted by $\mathcal{G}(c)$, is a continuous symmetric positive definite kernel. Moreover, the Dirac delta function $\delta(x) \in H^{-1}(\Omega)$ since $\Omega \subseteq \mathbb{R}$. We define $\mathcal{P}_i (u) = u(x_i)$, which is equivalent to $f_i(x)  = \delta (x-x_i)$ for $x_i \in \Omega$. We denote by $u_j(x)$ the PDE solution with the source term $f_j$,  $j=1,\ldots, n$.   In this case, our measurements 
\[
d_{ij} = \mathcal{P}_i(u_j) = u_j (x_i), \quad 1\leq i,j \leq n.
\]

We can verify that $\left(H_0^1(\Omega), \|\cdot \|_{\mathsf{U}}\right)$ is a reproducing kernel Hilbert space (RKHS) with the reproducing kernel being $\mathcal{G}(c)$. Thus, $v_j = q_j$ in~\eqref{eq:eFWI2}, where $v_j$ is the PDE solution in $\left(H_0^1(\Omega), \|\cdot \|_{\mathsf{U}}\right)$. In this example, the source operators are Dirac delta functions,  belonging to $H^{-1}(\Omega)$, which is the dual space of $H_0^1(\Omega)$.  As a result, for any fixed $j = 1,\ldots,n$, we can reformulate~\eqref{eq:eFWI_inner} as
\[
\min_{q \in \left(H_0^1(\Omega), \,  \|\cdot \|_{\mathsf{U}}\right)}   {\textstyle\frac{1}{2}}\sum_{i=1}^n |q(x_i) - e_{ij}|^2 +  {\textstyle\frac{\rho}{2}} \|q\|_{\mathsf{U}}^2.
\]
The optimal solution $q_j$ is
\[
q_j(x) = \sum_{k=1}^n \alpha_{jk} \mathcal{G}(c) (x,x_k) = \sum_{k=1}^n \alpha_{jk} \, u_k(x) ,
\]
where $\boldsymbol{\alpha}_{j} = (G(c) + \rho I)^{-1} \boldsymbol{e}_{j}$. The $ij$-th entry of the kernel matrix $G(c)$ is  
$$
g_{ij} = \langle u_i, u_j\rangle_{\mathsf{U}} =  \int_\Omega u_i(x) \mathcal{L}(c) \mathcal{G}(c)(x,x_j)\mathrm{d}x =u_i(x_j) = d_{ji}\,. 
$$
Note that $g_{ij} = g_{ji}$, and the symmetry comes from the fact that $\langle u, v\rangle_{\mathsf{U}} = \langle v, u\rangle_{\mathsf{U}}$. That is, we can obtain the Gram matrix $G$ directly from the data.

\begin{rem}
   In this example of the 1D elliptic equation, our method coincides with the Representer Theorem in the theory of RKHS. However, the connection with RKHS is no longer true for higher-dimensional elliptic equation examples since $\delta(x)$ is not in $H^{-s}$ for $s \leq d/2$ where $d$ is the dimension.
\end{rem}
\subsection{2D Elliptic equation}
Consider the 2D diffusion equation on a compact domain $\overline{\Omega}$ with a variable coefficient
\begin{equation}\label{eq:2d_poisson}
- \nabla \cdot \left( c(x) \nabla u(x) \right) = f(x),\quad x\in \Omega
\end{equation}
with the zero Dirichlet boundary condition on $\partial {\Omega}$. 
We further  assume that $0 < a \leq c(x) \leq b < \infty$, which implies that the PDE is uniformly elliptic and well-posed for the set of variable coefficients that we are interested in.

\subsubsection{Weak formulation}
The weak formulation is to find $u(x) \in \mathsf{U}$ such that
\begin{equation}\label{eq:forward_weak}
\mathcal{A}_c(u, \phi) = \mathcal{P}(\phi)\quad \forall \phi \in \mathsf{U},
\end{equation}
with
\[
\mathcal{A}_c(u,v) = \int_\Omega c(x)\, \nabla u(x) \cdot \nabla v(x) \,\mathrm{d}x\,,
\]
\[
\mathcal{P}(u) = \int_\Omega f(x) u(x)\, \mathrm{d}x.
\]
The weak formulation of the adjoint problem is to find $w \in \mathsf{U}$ such that
\begin{equation}\label{eq:adjoint_weak}
\mathcal{A}_c(\phi, w ) = \mathcal{A}_c ( w,\phi) = \mathcal{P}(\phi)\quad \forall \phi \in \mathsf{U},
\end{equation}
with $\mathcal{A}_c$ and $\mathcal{P}$ as before. Since the given elliptic operator is essentially self-adjoint, the adjoint problem~\eqref{eq:adjoint_weak} is exactly the same as the forward problem~\eqref{eq:forward_weak}.

\subsubsection{Data-driven kernel}
Similar to the 1D case,  we can define the inner product $\langle u, v \rangle_\mathsf{U}  = \mathcal{A}_c(u,v)$ which is a weighted $\dot{H}^1$ inner product by the positive variable coefficient $c(x)$. Then the Hilbert space $\mathsf{U}$  is given by
\[
\mathsf{U} = \left\{ u \, \left|\, \mathcal{A}_c(u,u) < \infty,
\, u|_{\partial \Omega} = 0\right.\right\},
\]
and its dual space is also defined accordingly. Then the weak forms \eqref{eq:forward_weak} and \eqref{eq:adjoint_weak} are well-posed.

We consider a finite number of sources $\{f_i\}$, $i=1,\ldots, n$, and define $p_i = \mathcal{R} f_i \in \mathsf{U}$, $\mathcal{P}_i(u) = \int_\Omega f_i \, u \rd x $, $1\leq i\leq n$. The function $f_i$ takes the form
\[
f_i(x) = \exp{(-20|x-x_i|^2)},\quad x\in \mathbb{R}^2\,,
\]
where $x_i$ is the center of the Gaussian. We use $n = 40$ sources and their location are depicted in Figure~\ref{fig:centers}.
\begin{figure}
\centering
\includegraphics[width = 0.5\textwidth]{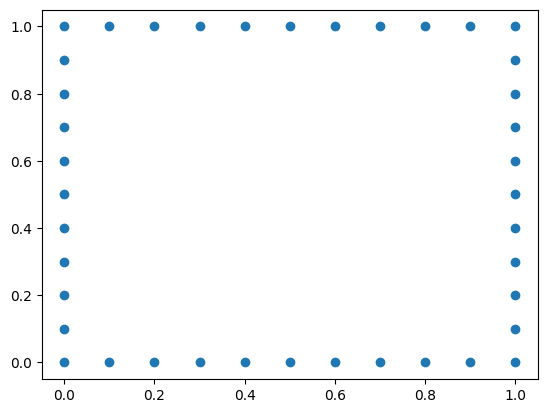}
    \caption{The location of the Gaussian centers for the direct method on the domain $[0,1]^2$.}\label{fig:centers}
\end{figure}
The corresponding PDE solution with the source $f_i$ is denoted by $u_i$. We first note that $w_i = u_i$ in~\eqref{eq:adjoint_weak} with $\mathcal{P}$ replaced by $\mathcal{P}_i$ since the PDE is self-adjoint. As a result, the Gram matrix is given by
\begin{equation}\label{eq:2d_poisson_kernel}
g_{ij}(k) = \langle u_i, u_j\rangle_{\mathsf{U}} = \int_\Omega c(x) \nabla u_i \cdot \nabla u_j  \,  \mathrm{d}x\ = \mathcal{A}_c(u_i,u_j) = \mathcal{P}_j(u_i) = d_{ji},\quad \forall i,j = 1,\ldots, n\,.
\end{equation}
That is, if $\{\mathcal{P}_i\}$ act as both the sources and the receivers, we can obtain the kernel matrix $G$ for the true coefficient through data measurements.

\subsubsection{Numerical example}
Next, we use an example to illustrate the impact of $\rho$ as well as the choice of inner product $\langle \cdot, \cdot \rangle_{\mathsf{U}}$ in the optimization landscape of the objective function $J$ in~\eqref{eq:new_obj} with respect to the coefficient $c$. Here, the domain $\Omega$ is a unit square $[0,1]^2$.

We consider a parameterized variable coefficient 
\[
c(x_1,x_2;\theta) = |\sin x_1|^2 + |\sin x_2|^2+ (1+100\,\theta)|\sin(10\,x_1)|^2 + |\sin(10\,x_2)|^2\,,\quad [x_1,x_2]^\top \in \Omega\,,
\]
where $\theta \in [0,2]$, the range of possible parameter values and the ground truth is $\theta^* = 0$. Note that $c$ is linear with respect to the parameterization $\theta$. Thus, the convexity of $J$ with respect to $c$ is represented here by the convexity of $J$ with respect to $\theta$.

First, we assume a finite-dimensional space $\mathsf{U} = \text{span}\{p_i\}_{i=1}^n$. That is, the solution space is a linear combination of the Riesz representation of the source terms $\{f_i\}$. The sources are located near the boundary of the domain.
In this setup, we consider two different inner products: a variable coefficient-independent one $\langle u, v\rangle_{\mathsf{U}} = \int_\Omega \nabla u \cdot \nabla v  \,  \mathrm{d}x$ and a variable coefficient-dependent inner product $\langle u, v\rangle_{\mathsf{U}} = \int_\Omega c(x) \nabla u \cdot \nabla v  \,  \mathrm{d}x = \mathcal{A}_c(u,v)$, for any $u, v $ in the Hilbert space $\mathsf{U}$. For the former inner product, the Riesz map $\mathcal{R}=-\Delta^{-1}$, while for the latter, $\mathcal{R}=\left(-\nabla\cdot (c(x) \nabla )\right)^{-1}$, both having the zero Dirichlet boundary condition.

We remark that, for the variable coefficient-dependent case,  the data-driven kernel formula~\eqref{eq:2d_poisson_kernel} still applies even though $\mathsf{U}$ here is an uncommon choice of finite-dimensional solution space.  We calculate the objective function $J(c)$ defined in~\eqref{eq:new_obj} for these two inner products, and examine the optimization landscape with respect to different values of $\rho$. Results are shown in Figure~\ref{fig:Poisson2D_source_basis} for these two different inner products.

Note that this example satisfies the assumptions in~\Cref{thm:quadratic}, so the statements apply. Based on Remark~\ref{rem:quadratic}, when the inner product does not depend on $c$ and $\rho = 0$, the objective function is quadratic with respect to the $c$, which can be seen in~\Cref{fig:Poisson2D_source_basis}(A). On the other hand, when $\langle \cdot, \cdot \rangle_{\mathsf{U}} = \mathcal{A}_c(\cdot, \cdot)$, $J_0$ is linear in $c$ as addressed in Remark~\ref{rem:linear} and illustrated in~\Cref{fig:Poisson2D_source_basis}(B). Moreover, with the $c$-independent inner product, the classic PDE-constrained optimization is known to be non-convex, corresponding to the case $\rho = \infty$ in~\Cref{fig:Poisson2D_source_basis}(A). Interestingly, this is no longer true if we use the $c(x)$-dependent inner product, which yields a convex objective function with respect to the coefficient $c(x)$. %

\begin{figure}[ht!]
    \centering
    \subfloat[$\langle \cdot, \cdot \rangle_\mathsf{U}$ without $c(x)$-dependency]{\includegraphics[width = 0.5\textwidth]{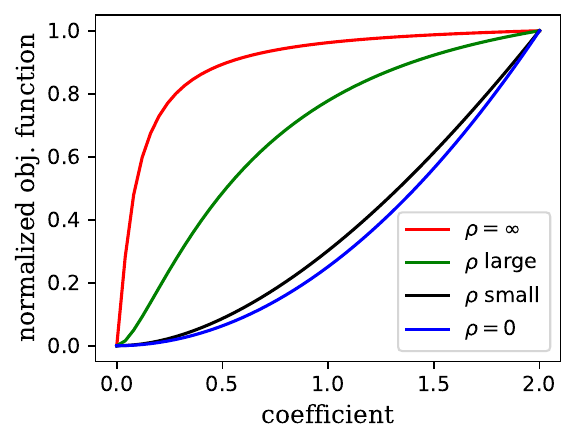}}
        \subfloat[$\langle \cdot, \cdot \rangle_\mathsf{U}$ with $c(x)$-dependency]{\includegraphics[width = 0.5\textwidth]{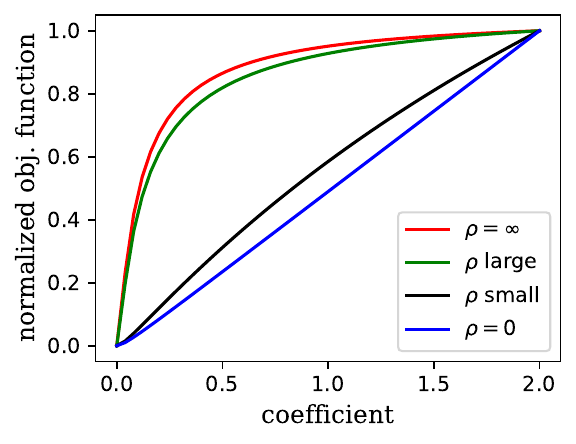}}
    \caption{The optimization landscape for PDE-constrained optimization based on the 2D inverse conductivity problem, for which the forward problem is the Poisson equation with a variable coefficient~\eqref{eq:2d_poisson}. The Hilbert space $\mathsf{U} = \text{span}\{p_i\}_{i=1}^n$. In (B), the inner product is given by $\langle u, v\rangle_{\mathsf{U}} = \int_\Omega \nabla u \cdot \nabla v  \,  \mathrm{d}x$ while the inner product in (B) is given by  $\langle u, v\rangle_{\mathsf{U}} = \int_\Omega c(x) \nabla u \cdot \nabla v  \,  \mathrm{d}x$.}\label{fig:Poisson2D_source_basis}
\end{figure}

Next, we repeat the example above except that the Hilbert space $\mathsf{U}$ is now  spanned by 2D first-order Lagrange elements over the entire domain $[0,1]^2$, a common choice in a low-order  finite-element method (FEM). Figure~\ref{fig:Poisson2D_full_basis} shows the optimization landscapes corresponding to various values of $\rho \in [0,\infty)$ and the two inner products. When $\rho =\infty$, the problem reduces to the classic PDE-constrained optimization based on the 2D Poisson equation with the squared $L^2$ norm to measure the data misfit. Since the measured solution is nonlinear in $c$, the resulting objective function is non-convex, as seen in Figure~\ref{fig:Poisson2D_full_basis}. As we gradually reduce $\rho$ from $\infty$ to $0$, the optimization landscape varies and eventually becomes a linear function for the $c$-dependent inner product, and a quadratic function for the $c$-independent inner product, both are advantageous for finding the global minimum $\theta^* = 0$. The  convex landscapes at $\rho = 0$ align well with those in Figure~\ref{fig:Poisson2D_source_basis} even though our theory, i.e., Corollary~\ref{thm:quadratic}, currently does not apply to this setup.

In both figures, we observe the gradual variance from the case $\rho = 0$ to the other limit $\rho = \infty$. These plots show the interpolative nature of the proposed norm~\eqref{eq:new_obj} for measuring the data misfit and the effectiveness in improving the optimization landscape when we incorporate the model and parameter properties into the objective function.

\begin{figure}[ht!]
    \centering
    \subfloat[$\langle \cdot, \cdot \rangle_\mathsf{U}$ without $c(x)$-dependency]{\includegraphics[width = 0.5\textwidth]{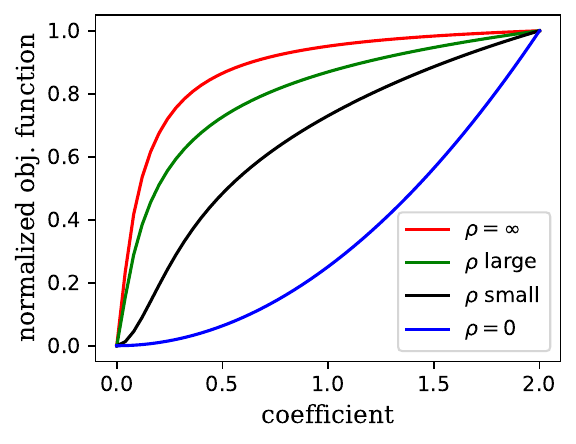}}
    \subfloat[$\langle \cdot, \cdot \rangle_\mathsf{U}$ with $c(x)$-dependency]{\includegraphics[width = 0.5\textwidth]{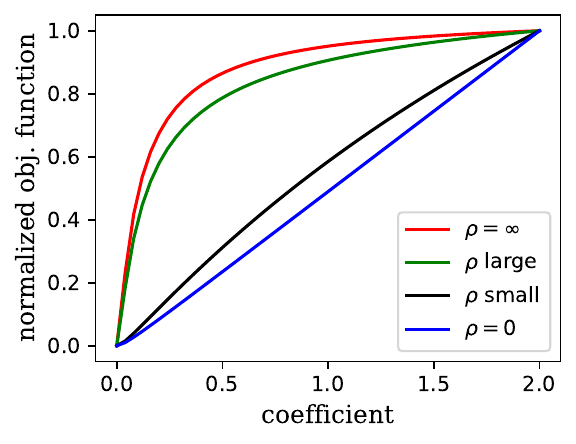}}
    \caption{PDE-constrained optimization based on the 2D inverse conductivity problem, for which the forward problem is the Poisson equation with variable coefficient~\eqref{eq:2d_poisson}. The Hilbert space $\mathsf{U}$ is spanned by 2D first-order Lagrange elements over the domain $[0,1]^2$.  In (A), the inner product is given by  $\langle u, v\rangle_{\mathsf{U}} = \int_\Omega \nabla u \cdot \nabla v  \,  \mathrm{d}x$. In (B), the inner product is given by $\langle u, v\rangle_{\mathsf{U}} = \int_\Omega c(x) \nabla u \cdot \nabla v  \,  \mathrm{d}x$. }\label{fig:Poisson2D_full_basis}
\end{figure}

\subsection{1D Helmholtz equation} Our next example is the 1D Helmholtz equation
\[
-u''(x;k) - (k / c(x))^2 u(x;k) = f(x),
\]
with boundary conditions
\[
u(0;k) = 0,
\]
\[
u'(1;k) - \imath \left(k/c(1)\right) u(1;k) = 0.
\]
We furthermore assume that $0 < a \leq c(x) \leq b < \infty$.

\subsubsection{Weak formulation}
The weak formulation is to find $u(\cdot\,; \,  k) \in \mathsf{U}$ such that
\begin{equation}\label{eq:weak_helm}
\mathcal{A}_{c,k}(u(\cdot\,;\, k), \phi) = \mathcal{P}(\phi)\quad \forall \phi \in \mathsf{U},
\end{equation}
with the bilinear form defined by
\[
\mathcal{A}_{c,k}(u,v) = \int_0^1 u'(x) \overline{v'(x)}\mathrm{d}x - k^2 \int_0^1 c(x)^{-2}u(x) \overline{v(x)}\mathrm{d}x -\imath k c(1)^{-1}u(1)\overline{v(1)},
\]
\[
\mathcal{P}(\phi) = \int_0^1 f(x)\overline{\phi(x)}\mathrm{d}x.
\]
The weak formulation of the adjoint problem is to find $w(\cdot \,;\, k) \in \mathsf{U}$ such that
\begin{equation}\label{eq:weakadjoint_helm}
\overline{\mathcal{A}_{c,k}(\phi, w(\cdot \,;\, k))} = \mathcal{P}(\phi)\quad \forall \phi \in \mathsf{U},
\end{equation}
with $\mathcal{A}_{c,k}$ and $\mathcal{P}$ as before. 

We define the inner product as $\langle u, v \rangle_\mathsf{U} = \int_0^1 u'(x)\overline{v'(x)}\mathrm{d}x$ and define the Hilbert space $\mathsf{U}$  as
\[
\mathsf{U} = \left\{ u : [0,1]\rightarrow\mathbb{C} \, \left|\, \|u\|_{\mathsf{U}} < \infty, u(0) = 0\right.\right\}.
\]
We then define its dual in the usual way. The the weak forms \eqref{eq:weak_helm} and \eqref{eq:weakadjoint_helm} are well-posed for $f \in H^{-1}$ \cite{Ihlenburg1997}.

The measurements are defined correspondingly by 
\[
d_{ij}(k) = \mathcal{P}_i(u_j(\cdot \,;\, k)) = \int_0^1 f_i(x) \overline{u_j(x;k)}  \rd x,
\]
where $i,j = 1,\ldots,n$, and $u_j(\cdot;k)$ is the solution to~\eqref{eq:weak_helm} with the source term $f_j$. Moreover, we assume $\{f_j\}_{j=1}^n$ are real-valued functions. Correspondingly, $w_j(\cdot;k)$ is the solution to~\eqref{eq:weakadjoint_helm} with the source term $f_j$, $j=1,\ldots,n$.

\subsubsection{Data-driven kernel}
We first note that $w_i(x;k) = \overline{u_i(x;k)}$ based on the properties of~\eqref{eq:weak_helm} and~\eqref{eq:weakadjoint_helm} (see~Lemma \ref{lma:hlmholtz1}), so  the Grammian $G(k) \in \mathbb{C}^{n\times n}$ (cf.~\Cref{thm:Representer}) has elements
\[
g_{ij}(k) = \langle u_j(\cdot;k), u_i(\cdot;k)\rangle_{\mathsf{U}}.
\]
We then find the following expression for the Grammian in terms of the data $d_{ij}(k)$ and the values of $b_i(k) = u_i(1;k)$ (see Lemma \ref{lma:hlmholtz2}):
\begin{equation}
\label{eq:data_gram_helm}
g_{ij}(k) = \Re\left(d_{ij}(k) + {\textstyle\frac{k}{2}}d_{ij}'(k)\right) + {\textstyle\frac{\imath k^2}{2c(1)}}\left(\overline{b_i'(k)}b_j(k) - \overline{b_i(k)}b_j'(k)\right),
\end{equation}
where $d_{ij}'$ and $b_i'$ denote the derivatives of $d_{ij}$ and $b_i$ with respect to $k$. Thus, we can compute the Grammian from the measurements if we have access to the response of the sources at $x=1$ and can compute the derivatives of the measurements with respect to $k$.

\subsubsection{Numerical example}
We consider a constant sound speed ($c(x)\equiv c$) and $\mathcal{P}_iu = u(x_i)$ (i.e., $f_i(x) = \delta(x-x_i)$). The solution in that case is given by the analytic formula 
\[
u_i(x; k) = 
\begin{cases} 
(k/c)^{-1}\sin((k/c) x)e^{\imath (k/c) x_i}
& x \leq x_i, \\
(k/c)^{-1}\sin((k/c) x_i)e^{\imath (k/c) x}
& x > x_i. \\
\end{cases} 
\]
This allows us to compute the Grammian $G$ in closed form and plot the objective functions resulting from the different formulations:
\begin{align}
J_\infty(c) &= {\textstyle\frac{1}{2}}\text{trace}\left(E(c)^*E(c)\right),\label{conventional}\\
J_\rho(c) &= {\textstyle\frac{1}{2}}\text{trace}\left(E(c)^*\left(I + \rho^{-1}G(c)\right)^{-1}E(c)\right),\label{variable metric}\\
\widetilde{J}_\rho(c) &= {\textstyle\frac{1}{2}}\text{trace}\left(E(c)^*\left(I + \rho^{-1}\widetilde{G}\right)^{-1}E(c)\right).\label{data-driven metric}
\end{align}
In this notation, the wavenumber $k$ is implicit and fixed. The fundamental difference between~\eqref{variable metric} and~\eqref{data-driven metric} is that the former uses a parameter-dependent Grammian, which we denote by $G(c)$ in a slight abuse of notation, while the latter uses the data-driven Grammian defined in~\eqref{eq:data_gram_helm}, which here we denote by $\widetilde{G}$. 

In this experiment, we take measurements at $x_i = i / (n+1)$ for $i = 1, \ldots, n$ and generate (noiseless) data for $c = 1$. The behavior of the three objectives for various values of $(n, k, \rho)$ is shown in Figures \ref{fig:helmholtz1da}-\ref{fig:helmholtz1dc}.

We highlight a few interesting observations. First, the optimization problem generally gets more challenging (with more local minima) for larger $k$ (higher frequencies). This is expected and well-known behavior for inverse problems with the Helmholtz equation. Second, the relaxed approach for small $\rho$ makes the problem easier to solve (i.e., with fewer local minima, which is advantageous for gradient-based optimization algorithms) when more measurements are available (following Theorem \ref{thm:quadratic}). Third, the data-driven metric formulated in~\eqref{data-driven metric} is a good alternative to the variable metric approach given in~\eqref{variable metric} in these examples. Note that we do not need to differentiate the Grammian using the data-driven metric, but only the data misfit $E(c)$, which can be computed using the adjoint-state method. Therefore, using any gradient-based optimization algorithms, such as steepest descent and nonlinear conjugate gradient methods, the objective function \eqref{data-driven metric} incurs the same computational cost compared to the objective function~\eqref{conventional}. %
Finally, we note that for $n=10$ we observe the convex behavior suggested by section \ref{sec:convex} as only in the limit of large $n$ does the finite-dimensional space yield a sufficiently accurate representation of the solution of the PDE.
\begin{figure}
\includegraphics[scale=.5]{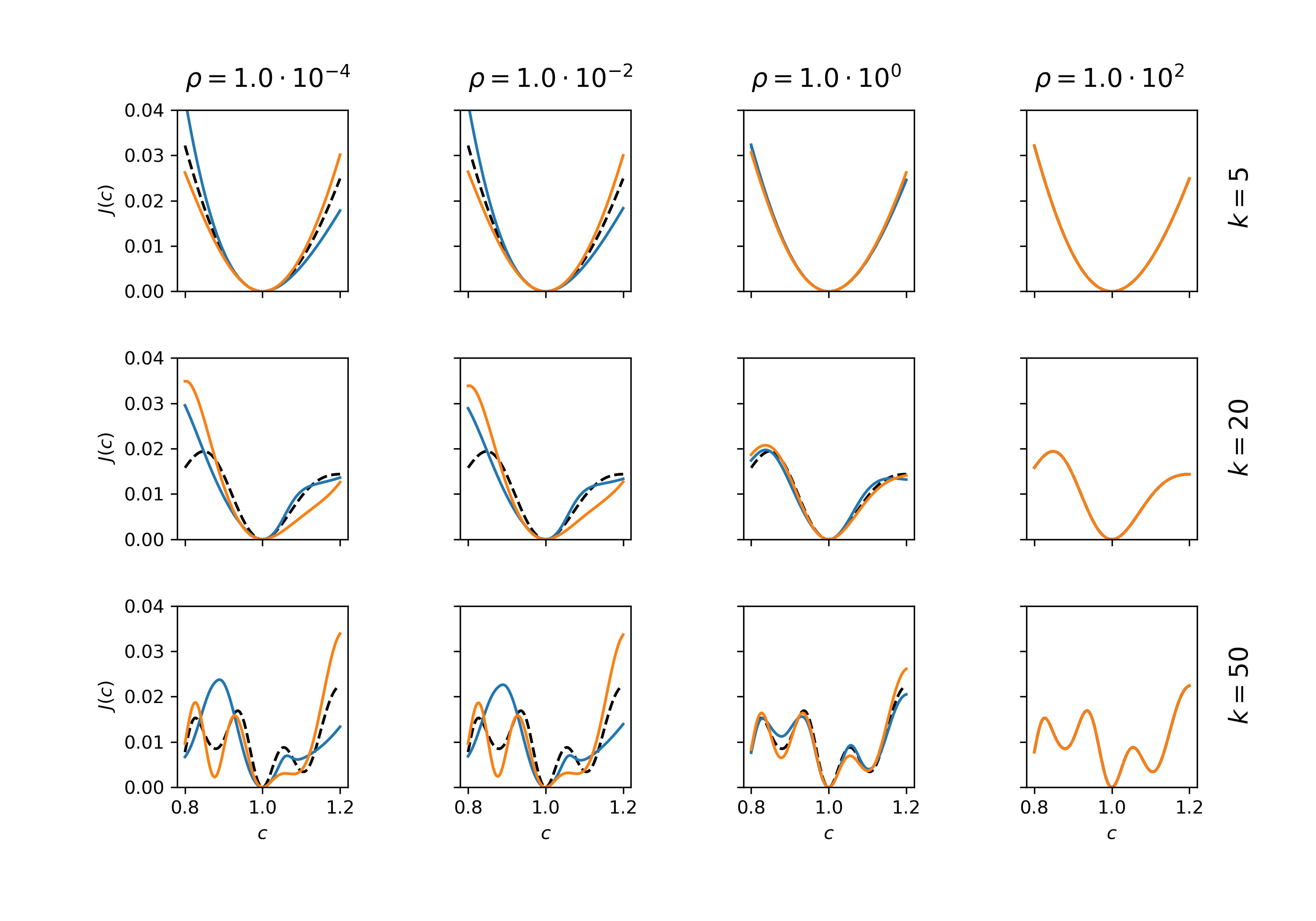}
\caption{Helmholtz example: different objective functions for $n=2$; conventional (dashed), variable metric (blue), data-drive metric (orange).}\label{fig:helmholtz1da}
\end{figure}

\begin{figure}
\includegraphics[scale=.5]{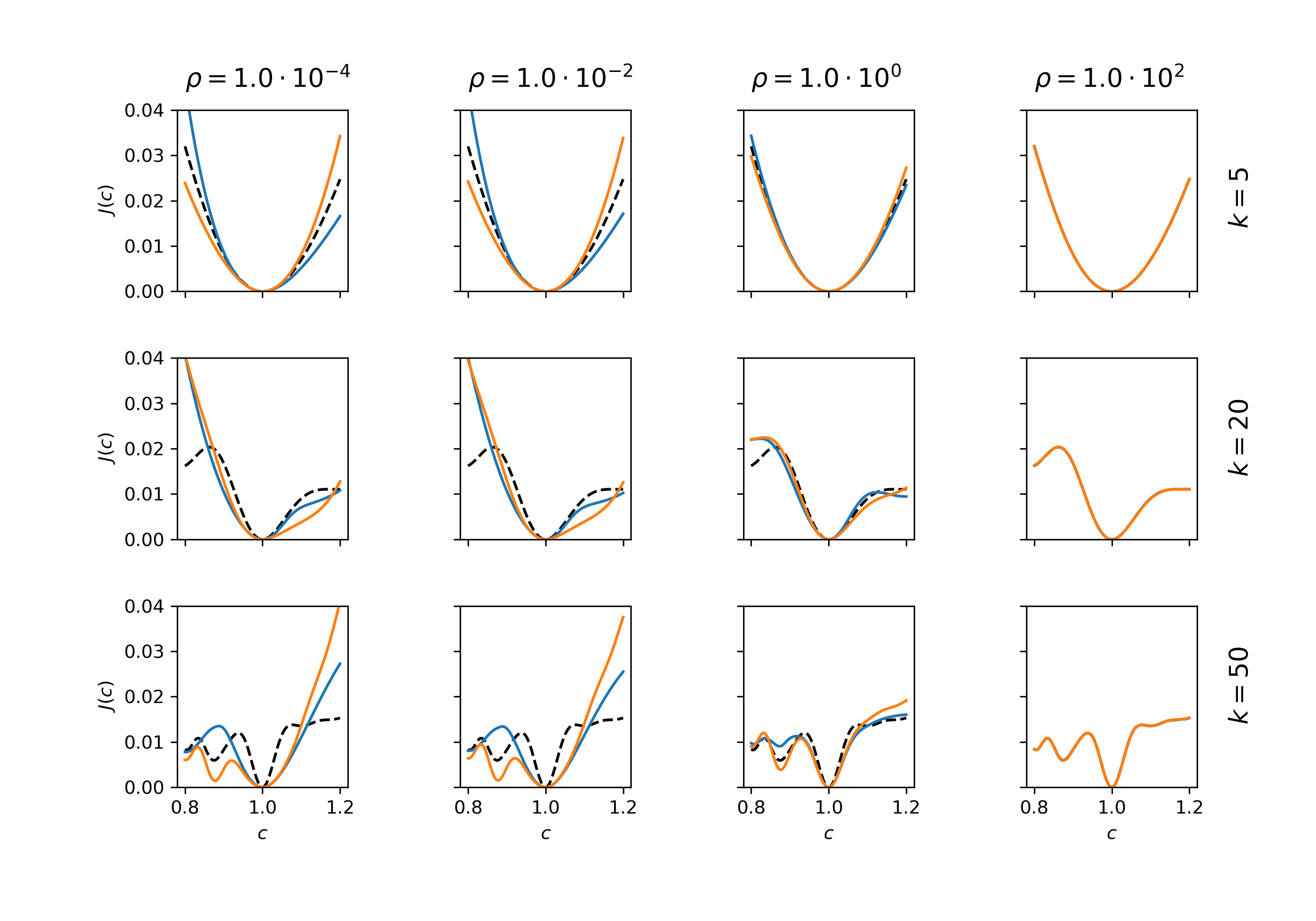}
\caption{Helmholtz example: different objective functions for $n=5$; conventional (dashed), variable metric (blue), data-drive metric (orange).}\label{fig:helmholtz1db}
\end{figure}

\begin{figure}
\includegraphics[scale=.5]{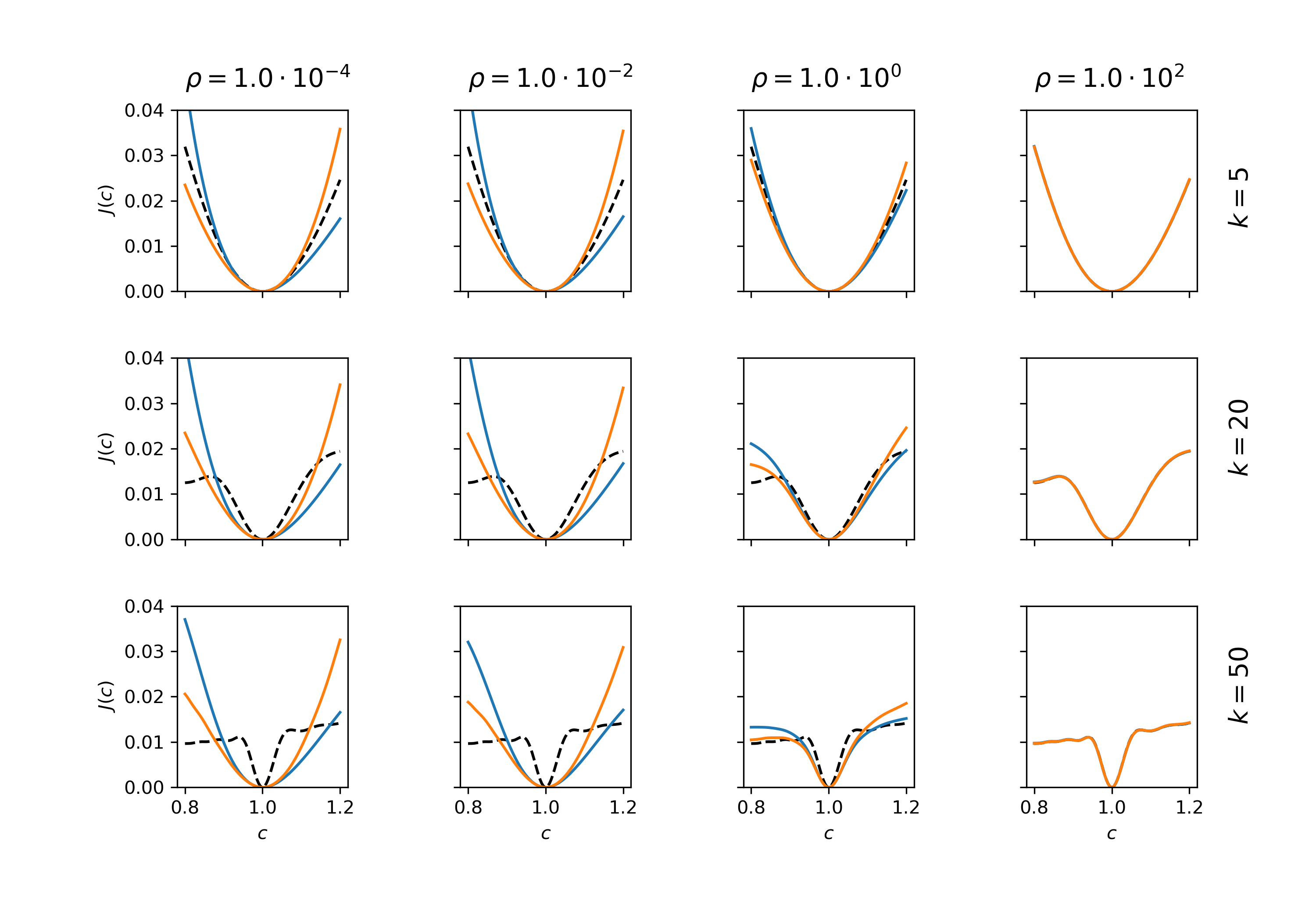}
\caption{Helmholtz example: different objective functions for $n=10$; conventional (dashed), variable metric (blue), data-drive metric (orange).}\label{fig:helmholtz1dc}
\end{figure}
\subsection{Schr{\"o}dinger equation}
We consider the following Schr{\"o}dinger equation
\begin{equation}\label{eq:sch_eqn}
-\nabla^2 u(x;\lambda) + c(x)u(x;\lambda) - \lambda u(x;\lambda) = f(x), \quad x \in \Omega\subset{\mathbb{R}^d},
\end{equation}
with homogeneous Dirichlet boundary conditions. The inverse problem of estimating the scattering potential $c(x)$ from measurements of $u$ for various source term $f(x)$ and different values of $\lambda$ is often studied in the context of inverse scattering \cite{imanuvilov2012inverse,Novikov2022}.

\subsubsection{Weak formulation}
To guarantee the well-posedness of the forward problem, it is natural to consider the function space $\mathsf{U} = H_0^1(\Omega)$, so the forward problem well-posed for $f_i \in H^{-1}(\Omega)$~\cite{Victor2014}.

We first define the bilinear form
$$\mathcal{A}_{c,\lambda}(u,v)=\int_{\Omega} \nabla u(x)\cdot \nabla v(x)\mathrm{d}x + \int_{\Omega} c(x)u(x)v(x)\mathrm{d}x - \lambda \int_{\Omega} u(x)v(x)\mathrm{d}x,$$
$$\mathcal{P}(u) = \int_{\Omega} f(x)u(x)\mathrm{d}x.$$
We consider $n$ source functions $\{f_i\}_{i=1}^n \subset \U^*$, whose Riesz representations are $\{p_i\}_{i=1}^n \subset \mathsf{U}$. 
The measurements are then denoted by
\[
d_{ij}(\lambda) = \mathcal{P}_i(u_j(\cdot\,;\,\lambda)) = \langle p_i, u_j(\cdot \,;\,\lambda)\rangle_\mathsf{U},
\]
where $u_i(\cdot\,;\,\lambda)$ is the (weak) solution of the Schr{\"o}dinger equation for frequency $\lambda$ and source term $f_i$. Due to symmetry we immediately find $d_{ij}(\lambda)=d_{ji}(\lambda)$, $i,j=1,\ldots,n$.  
We will first consider this choice with a $c(x)$-weighted inner product to be able to derive a data-driven Gram matrix $G$. Alternatively, we can also use the usual (unweighted) $H^1_0(\Omega)$ inner product and consider a finite-dimensional subspace spanned by the Riesz representations $\{p_i\}_{i=1}^n \subset H_0^1(\Omega)$ of $f_i \in H^{-1}(\Omega)$ to derive a direct inversion method according to Theorem~\ref{thm:quadratic}.

\subsubsection{Data-driven kernel}
To derive an expression that relates the kernel to the measurements, consider the weighted inner product:
$$\langle u,v\rangle_{\mathsf{U}} = \int_{\Omega} \nabla u(x)\cdot\nabla v(x)\mathrm{d}x + \int_{\Omega} c(x)u(x)v(x)\mathrm{d}x.$$
Since the equation~\eqref{eq:sch_eqn} is self-adjoint, we get $w_i = u_i$ for $i=1,\ldots,n$, and hence the $ij$-th entry of the Gram matrix is given by
$$g_{ij}(\lambda) = \langle u_i(\cdot \, ;\, \lambda), u_j(\cdot\,;\,\lambda)\rangle_\mathsf{U}.$$
We will use the short-hand notation $u_i(\lambda) = u_i(\cdot \, ;\, \lambda)$. From the weak form, we find that 
$$\langle u_i(\lambda), u_j(\mu)\rangle_{\mathsf{U}} - \lambda\langle u_i(\lambda),u_j(\mu)\rangle_{L^2(\Omega)} = d_{ij}(\mu),$$
$$\langle u_i(\lambda), u_j(\mu)\rangle_{\mathsf{U}} - \mu\langle u_i(\lambda),u_j(\mu)\rangle_{L^2(\Omega)} = d_{ij}(\lambda).$$
Using these two relations, we can rewrite the $ij$-th entry of the Gram matrix as
$$G_{ij}(\lambda)=\langle u_i(\lambda), u_j(\lambda)\rangle_{\mathsf{U}} = \lim_{\mu\rightarrow\lambda}\frac{\mu d_{ij}(\mu) - \lambda d_{ij}(\lambda)}{\mu-\lambda}.$$
This simplifies to
$$
g_{ij}(\lambda) = \lim_{h\rightarrow 0}\frac{(\lambda + h)d_{ij}(\lambda+h) - \lambda d_{ij}(\lambda)}{h} = d_{ij}(\lambda)+\lambda d_{ij}'(\lambda).
$$
To calculate $G_{ij}(\lambda)$,  we thus need to measure the derivative of the measurements with respect to the hyperparameter $\lambda$ as well, which can only be computed when we have measurements with densely varying $\lambda$ so that a divided difference approximation yields a reasonably good result. %

\subsubsection{A direct method}
To illustrate the effectiveness of the direct method in Remark~\ref{rem:direct}, we consider the $c$-independent inner product: $\langle u, v\rangle_{\mathsf{U}} = \int_{\Omega} \nabla u(x)\cdot\nabla v(x) \mathrm{d}x$ and define $p_i$ as the solution of $- \nabla^2 p_i = f_i$ with the zero Dirichlet boundary conditions. That is, $p_i$ is the Riesz representation of $f_i$ given this inner product. Moreover, we assume that $\mathsf{U} = \text{span}\{p_1,\ldots, p_n\}$, and the domain $\Omega = [0,1]^2 \subset \mathbb{R}^2$.

\begin{figure}
    \centering
    \subfloat[The true coefficient]{\includegraphics[width = 0.49\textwidth]{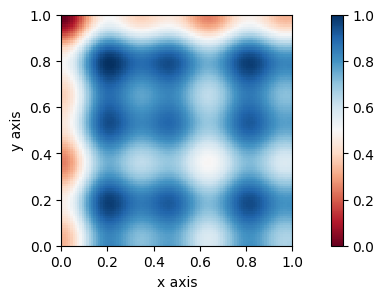}}
    \subfloat[Inversion with  $a = 1$]{\includegraphics[width = 0.49\textwidth]{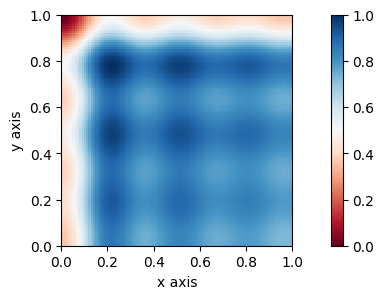}}\\
    \subfloat[Inversion with  $a=3$]{\includegraphics[width = 0.49\textwidth]{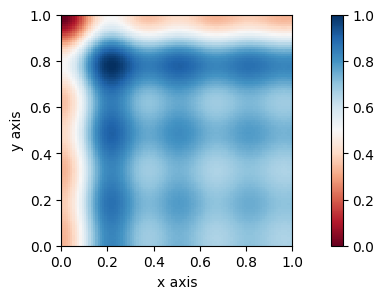}}
    \subfloat[Inversion with $a = 10$]{\includegraphics[width = 0.49\textwidth]{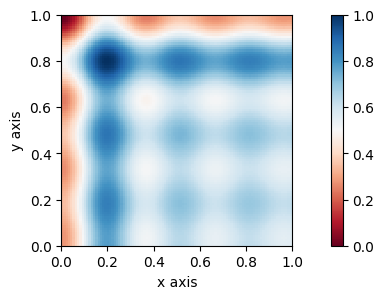}}
    \caption{A direct method~\eqref{eq:direct} to reconstruct the variable coefficient for the linear Schr\"odingr equation. (A): the true coefficient $c(x)$; (B)-(D): the direct reconstruction using the Gaussian sources of different parameters,  ranging from $a=1$ (B), $a=3$ (C) to $a=10$ (D).}
    \label{fig:direct_method}
\end{figure}

We can now apply the approach suggested by Corollary~\ref{thm:quadratic} and Remark~\ref{rem:direct}. Expressing the scattering potential $c(x)$ in terms of a basis $\{\psi_k\}_{k=1}^{n'}$, i.e., $c(x) = \sum_{k=1}^{n'} c_k \psi_k(x)$, we get the following set of $n^2$ linear equations for its coefficients $\{c_k\}_{k=1}^{n'}$:
\begin{equation}\label{eq:direct}
\sum_{k=1}^{n'} h_{ijk} \, c_k = -m_{ij} + \lambda s_{ij} + \left(MD(\lambda)^{-1}M\right)_{ij}, \, i,j = 1, 2, \ldots, n,
\end{equation}
with
\[
h_{ijk} = \langle \psi_k p_i\,, \,p_j\rangle_{L^2(\Omega)}\,, \quad m_{ij} = \langle p_i\,, \,p_j\rangle_{\mathsf{U}}, \quad s_{ij} = \langle p_i\,, \, p_j\rangle_{L^2(\Omega)}.
\]
Note that if $n^2 \gg n'$, the linear system is severely over-determined. We may use a subset of the $n^2$ equations instead.

The basis function for the variable coefficient $c(x)$, $\{\psi_k\}$, is set to be
\[
\psi_k(x) = |\sin ( k x )|^2, \quad x\in \mathbb{R}^2\,,\quad k = 1,2,\ldots, 10\,.
\]
The reference coefficients $\{c_k\}_{k=1}^{10}$ are randomly drawn. In Figure~\ref{fig:direct_method}(A), we present the groudtruth for reference. We set up the tests such that the source function $f_i$ is a Gaussian centered at location $x_i$ with different width $a>0$:
\[
f_i(x) = \exp{(-a|x-x_i|^2)},\quad x\in \mathbb{R}^2\,,\quad i = 1,\ldots,40\,.
\]
The location of the Gaussian centers $\{x_i\}_{i=1}^{40}$ is fixed for all sets of basis; see Figure~\ref{fig:centers} for an illustration. The basis functions $\{p_i\}_{i=1}^{40}$ are the Riesz representation of $f_i$ in $\mathsf{U}$ with respect to the $\dot{H}^1$ inner product.

In Figure~\ref{fig:direct_method}(B), (C), and (D), we use $40$ basis functions, but with the Gaussian parameter $a$ varying from $1$, $3$ to $10$, respectively. Since $n = 40$ and $m = 10$, the problem is very over-determined. We fix $j=1$, and consider $i=1,\ldots, 10$ in~\eqref{eq:direct}. Thus, the linear system to be solved is fully determined. The set of basis functions with the largest coverage yields the best result (see Figure~\ref{fig:direct_method}(B)), and the inversion worsens as the variance of the Gaussian basis in defining the $\{f_i\}$ decreases. This is to be expected since the larger the variance of the $\{f_i\}$ located on the boundary, the more overlap between their Riesz representations $\{p_i\}$ and the variable coefficient $c(x)$ defined on the interior of the domain. 

\subsection{Seismic inversion with 2D Helmholtz equation}

Our final example is seismic inversion constrained by the 2D Helmholtz equation on a half-space:
\[
\nabla^2 u(x) + k^2 c(x)^{-2} u(x) = f(x)\,,\quad x\in \Omega \subset \mathbb{R}^2
\]
with absorbing boundary conditions. Given a sequence of sources $\{f_i\}_i$ located on the upper boundary of the domain, the inverse problem aims to recover the wave propagation speed $c(x)$ which represents the property of the subsurface material. In our tests, we use a rectangular domain of size $3 \times 11$ km and collect data for $124$ equally spaced sources / receivers across the $x$-axis from $0.1$km to $10$km at the depth of $0.04$ km in the $z$ axis at a single frequency (to be specified below). The true coefficient is show in Figure~\ref{fig:marm-vels}(a).

To test the robustness of the inversion process to initialization, we consider two experimental setups with different initial guesses. In the first case, the initial velocity is a smoothed version of the ground-truth velocity, while in the second case, the initial velocity is a linearly increasing profile,  see Figure~\ref{fig:marm-vels} (b-c). For each experiment, we perform PDE-constrained optimization with the three objective functions in~\eqref{conventional}, \eqref{variable metric} and~\eqref{data-driven metric}, respectively. These three objective functions correspond to the conventional least-squares method,  the least-squares method with a variable metric, and the least-squares method with the data-driven metric. In both cases, the hyper-parameter $\rho$ in the objective function with a variable metric~\eqref{variable metric} and the one with a data-driven metric~\eqref{data-driven metric} are fixed to be $1$ throughout. To solve the resulting optimization problems we use L-BFGS and run to convergence.

\begin{figure}
\centering
\begin{tabular}{ccc}
\includegraphics[width=0.33\linewidth]{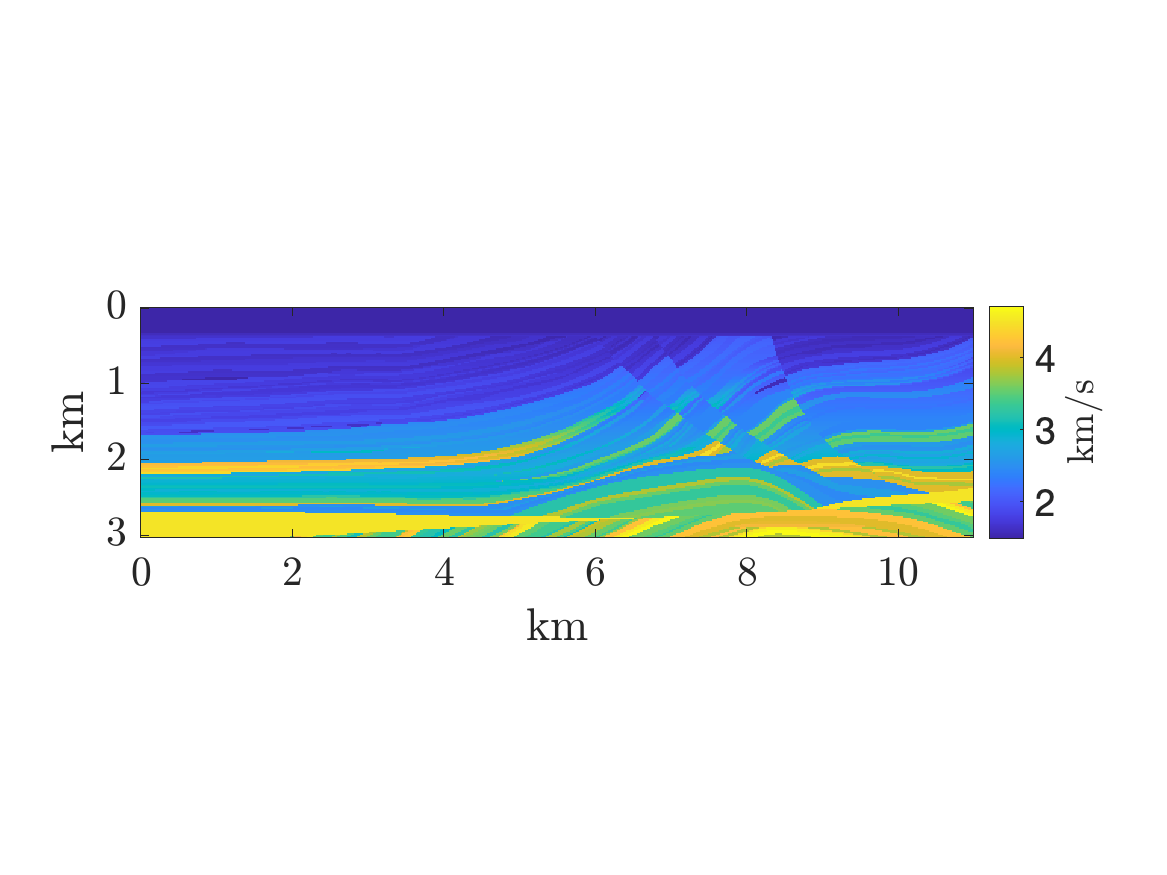}&
\includegraphics[width=0.33\linewidth]{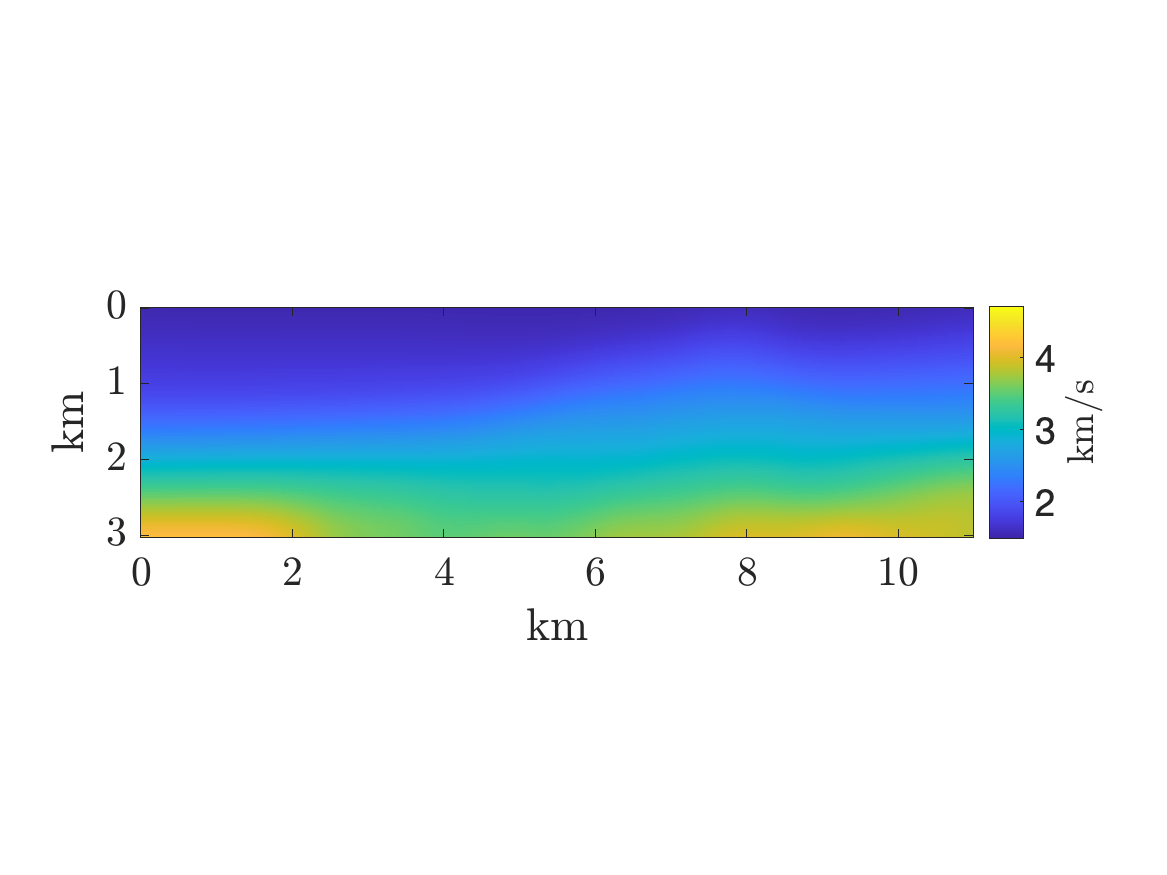}&
\includegraphics[width=0.33\linewidth]{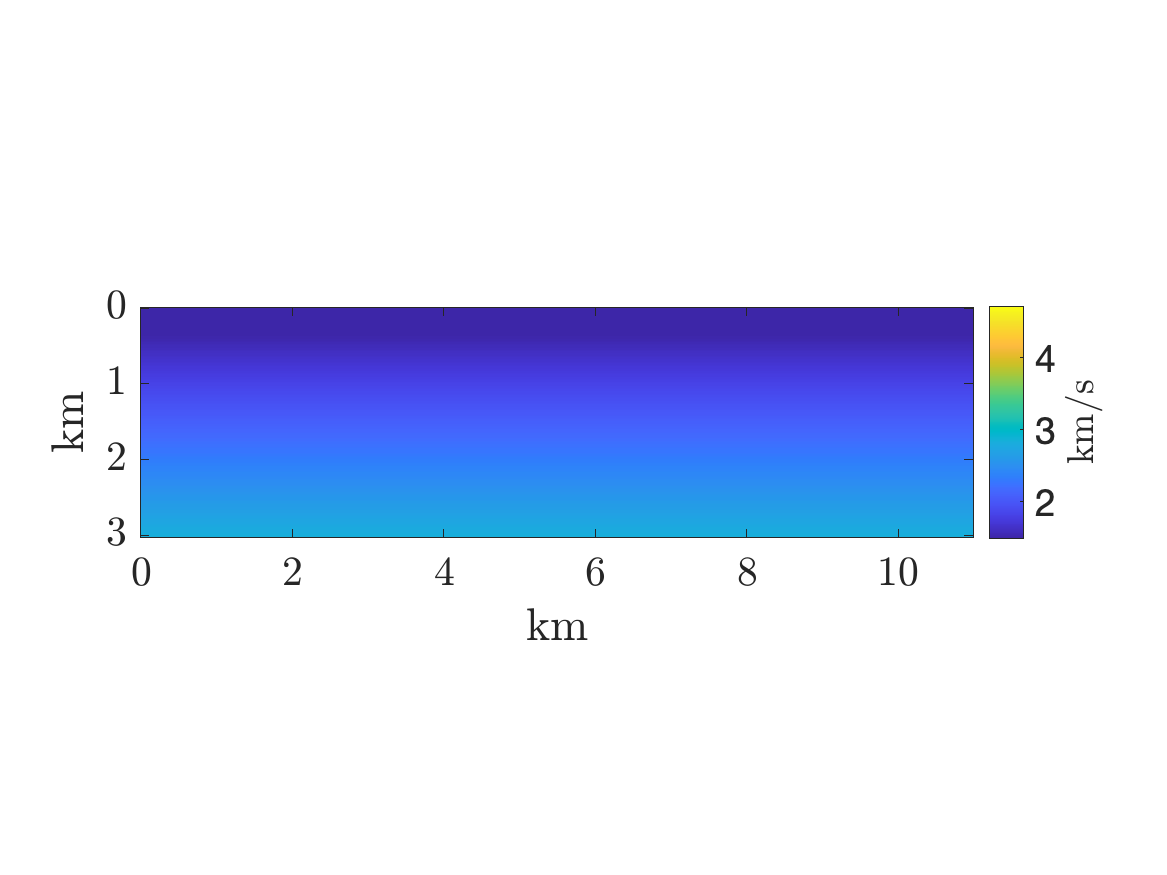}\\
(a)&(b)&(c)\\
\end{tabular}
\caption{Left: ground-truth velocity; middle: initial velocity for Case 1; right: initial velocity for Case 2.}
\label{fig:marm-vels}
\end{figure}

We run the following experiments:
\begin{itemize}
\item[Experiment 1] Here, we set the frequency to be $4$hz, the initial velocity as depicted in Figure~\ref{fig:marm-vels}(b). The top row of Figure~\ref{fig:smooth-vels} shows the reconstructed velocities. The bottom row compares the actual observed and simulated data at receivers from one source. 
\item[Experiment 2] Here, we set the frequency to be $6$hz, the initial velocity as depicted in Figure~\ref{fig:marm-vels}(c). The top row of Figure~\ref{fig:line-vels} shows the reconstructed velocities. The bottom row compares the actual observed and simulated data at receivers from one selected source. 
\end{itemize}

From the numerical results in Figures~\ref{fig:smooth-vels} and~\ref{fig:line-vels}, we see that inversions using the objective function with a variable metric~\eqref{variable metric} and the one with a data-driven metric~\eqref{data-driven metric} achieve better data fitting. At the same time, the conventional method~\eqref{conventional} shows an evident local minimum trapping. Moreover, the reconstructed velocities under~\eqref{variable metric} and~\eqref{data-driven metric} have a better recovery in the deeper part of the velocity (when $z$ is large) with more low-wavenumber components recovered. Even though the weightings in the objective functions apply to the data side, they also consequently scale the velocity components during inversion. Although our theory part in Section~\ref{sec:theory} is based on simple setups with assumptions, this set of tests sheds light on the potential of variable-dependent metric and data-driven metric as an objective function in mitigating local minima in PDE-constrained optimizations.

\begin{figure}
\centering
\includegraphics[width=0.33\linewidth]{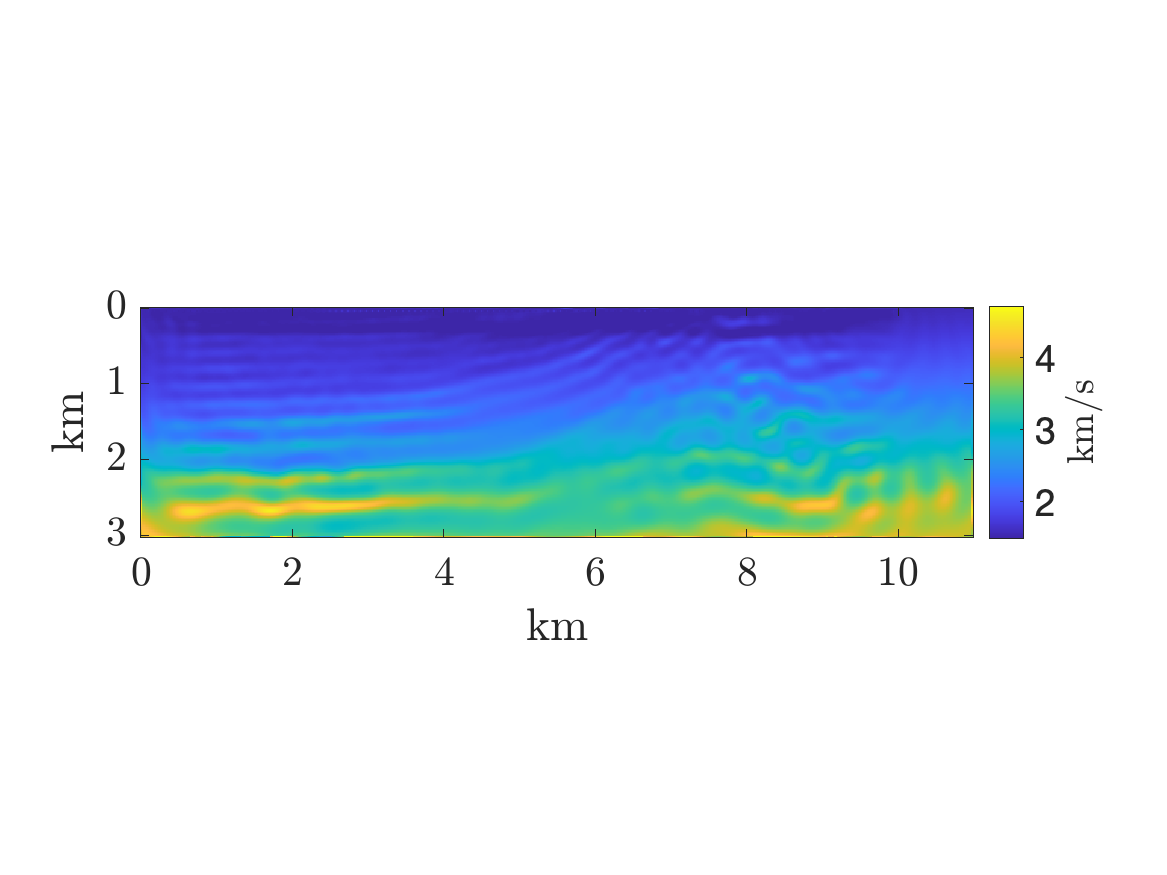}
\includegraphics[width=0.33\linewidth]{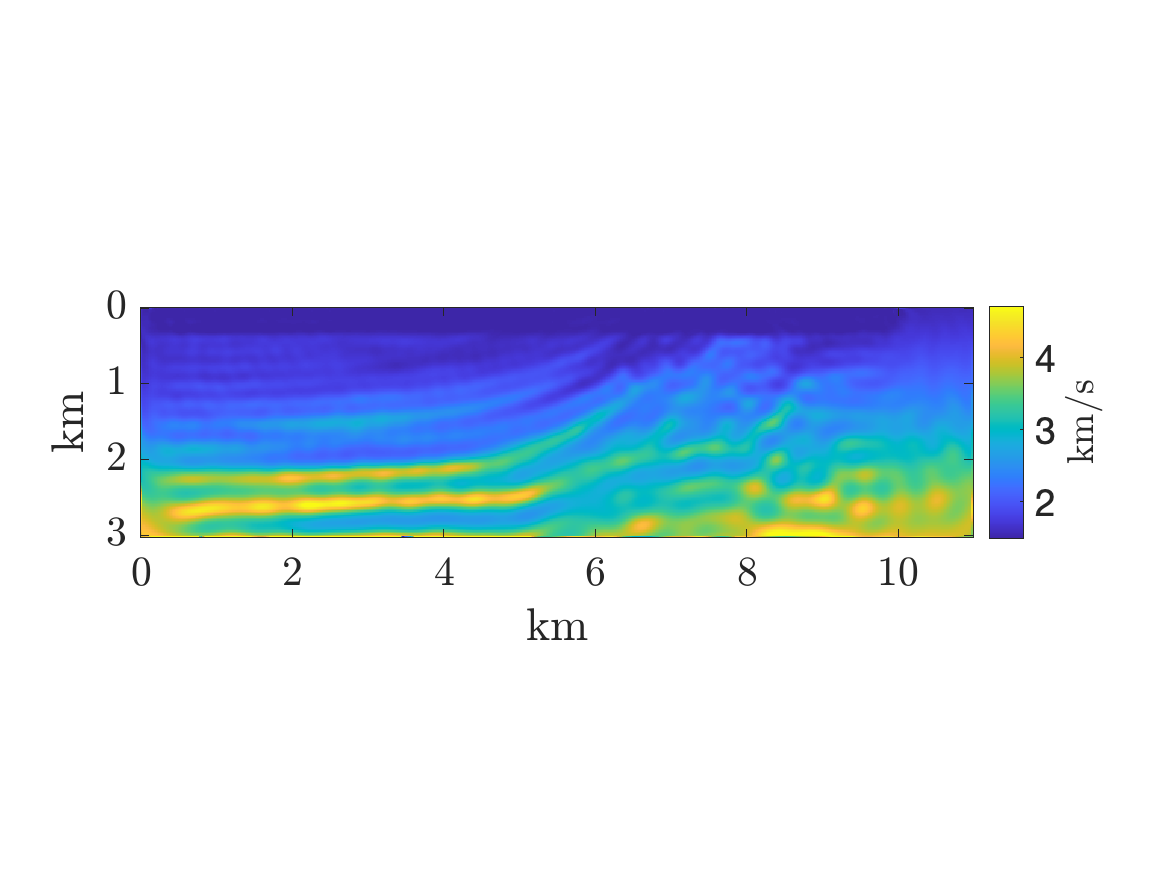}\includegraphics[width=0.33\linewidth]{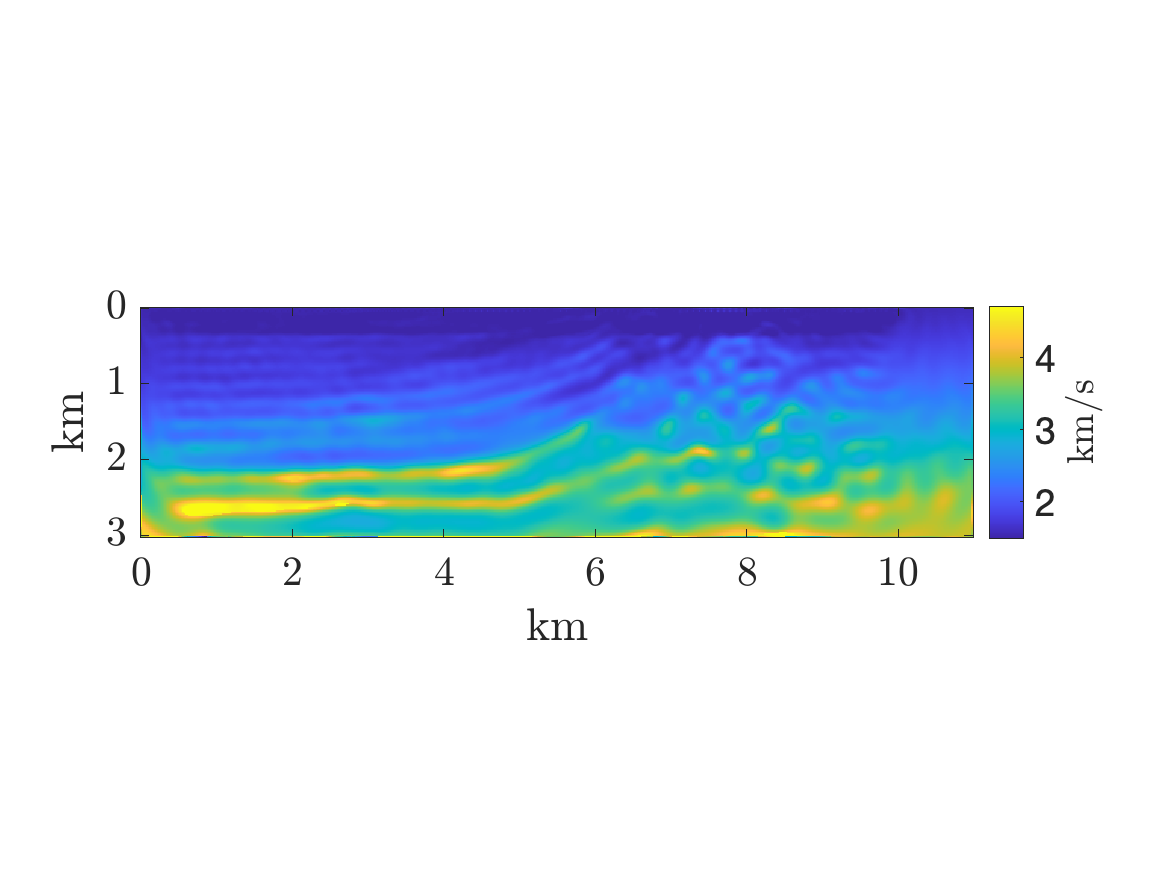}\\
\includegraphics[width=0.33\linewidth]{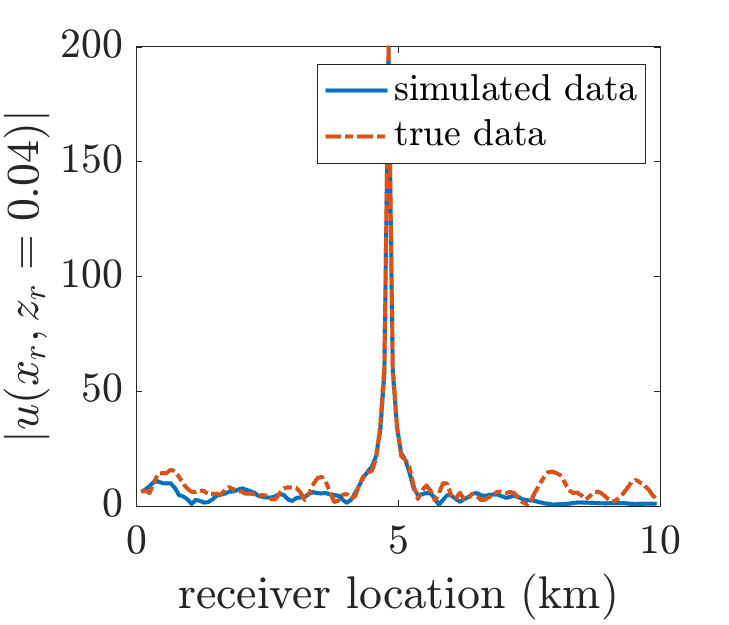}
\includegraphics[width=0.33\linewidth]{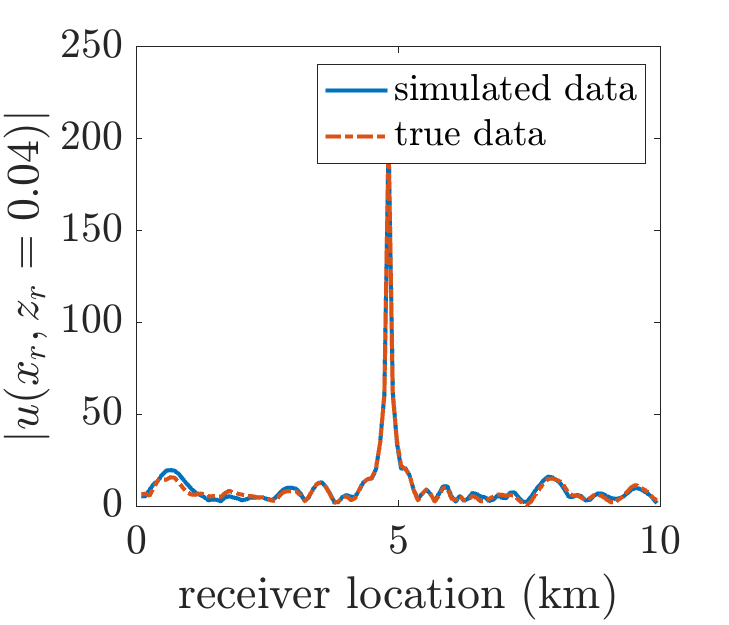}\includegraphics[width=0.33\linewidth]{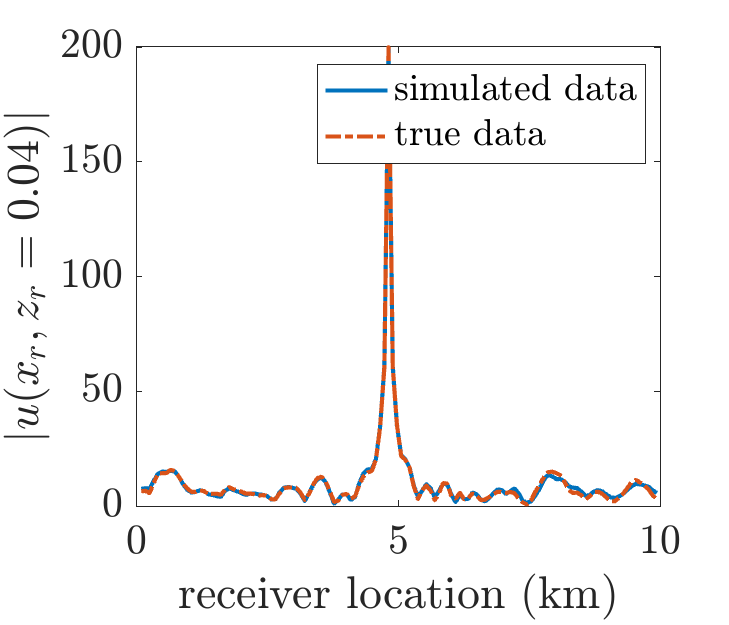}
\caption{The top row shows the reconstructed velocities. The bottom row compares simulated and true data from one source. From left to right shows the result using the objective function~\eqref{conventional}, \eqref{variable metric}, and~\eqref{data-driven metric}, respectively. This sets of result correspond to Case 1, i.e., the initial velocity is the middle panel of Figure~\ref{fig:marm-vels}.}
\label{fig:smooth-vels}
\end{figure}

\begin{figure}
\centering
\includegraphics[width=0.33\linewidth]{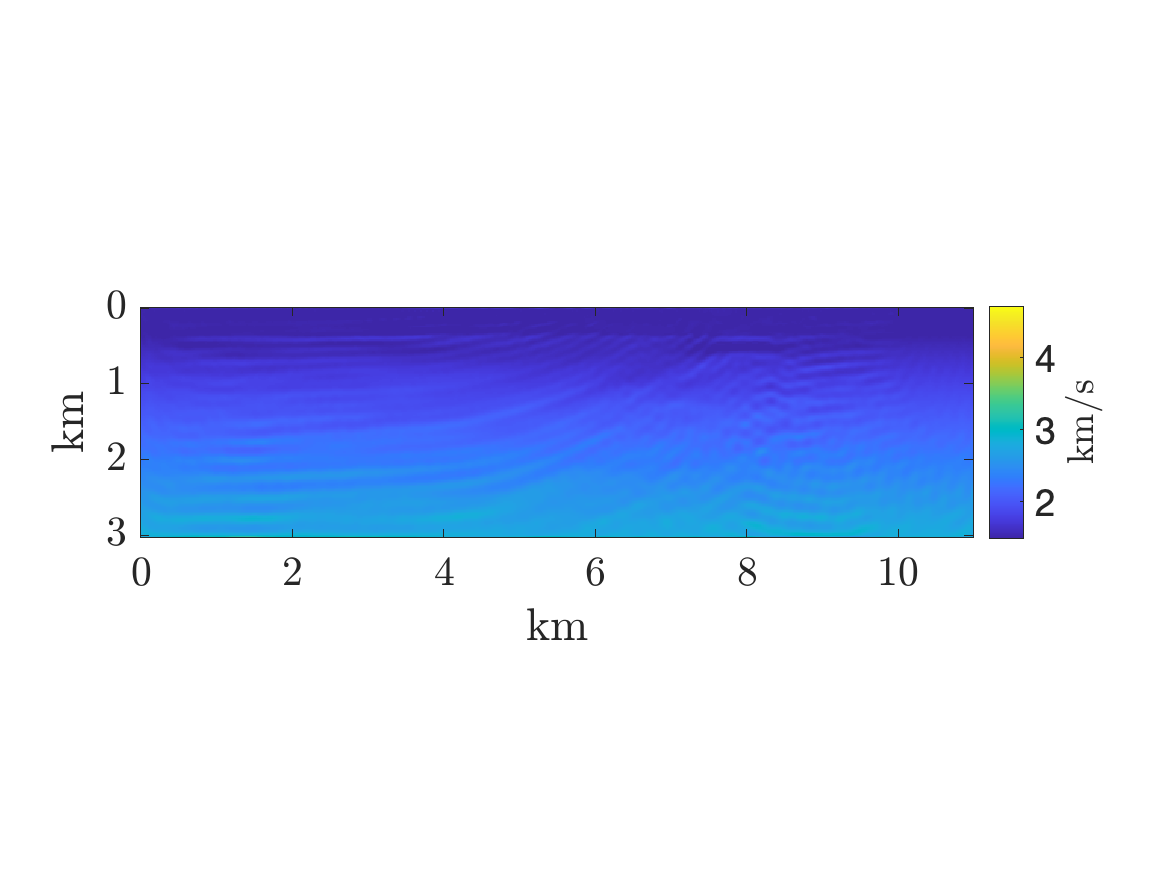}
\includegraphics[width=0.33\linewidth]{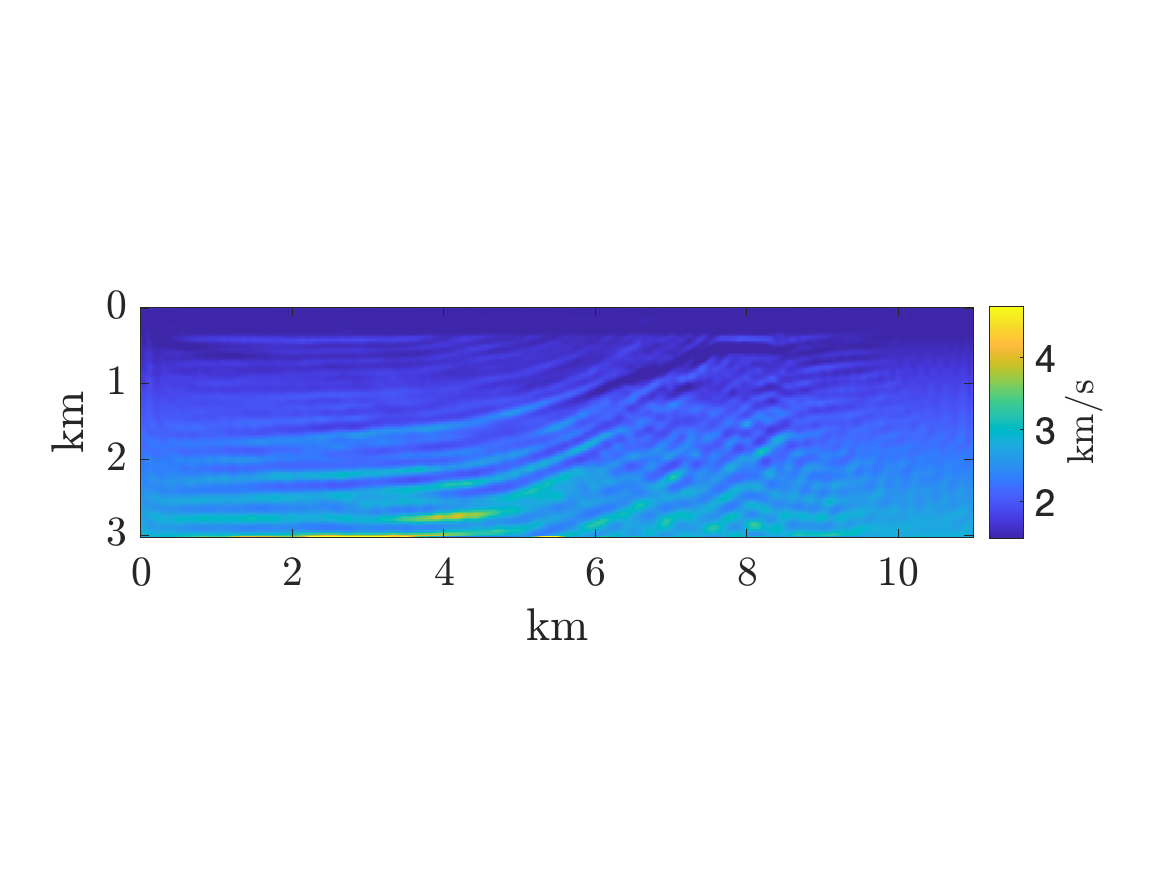}\includegraphics[width=0.33\linewidth]{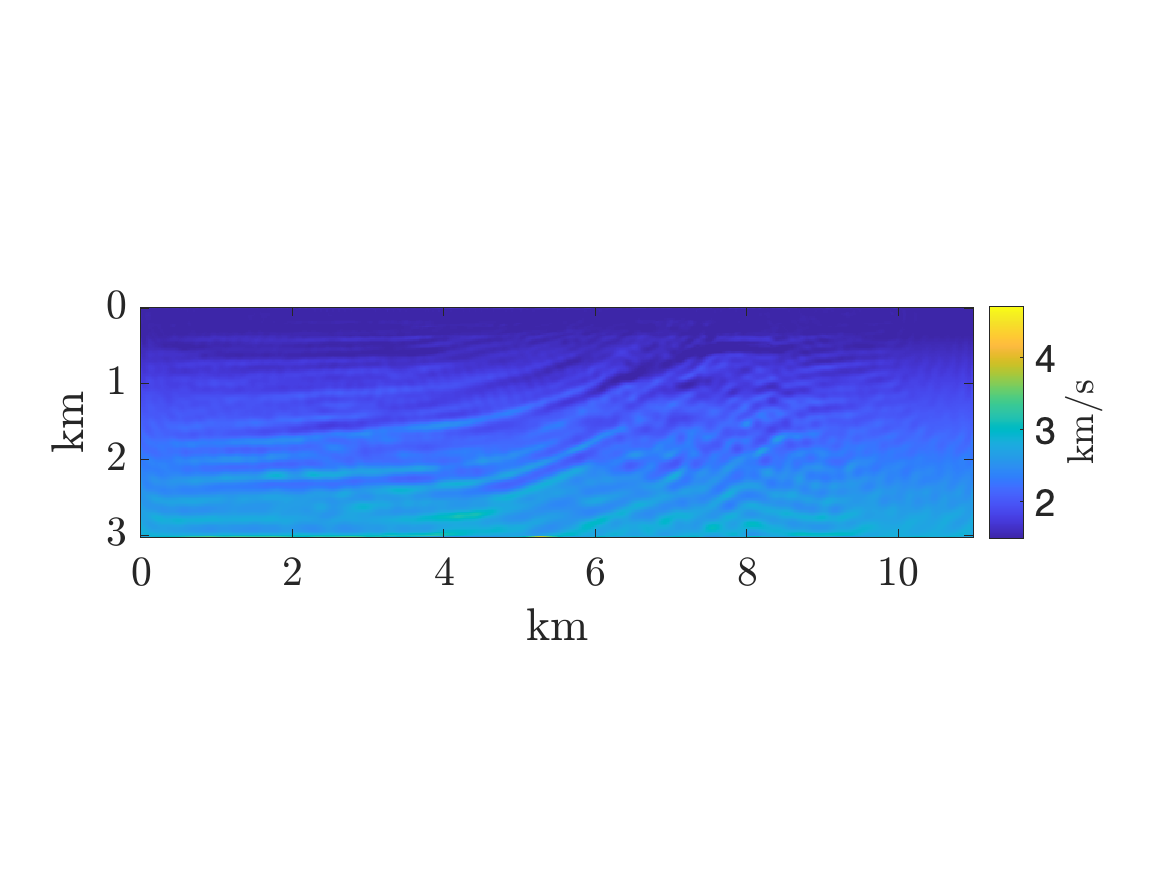}\\
\includegraphics[width=0.33\linewidth]{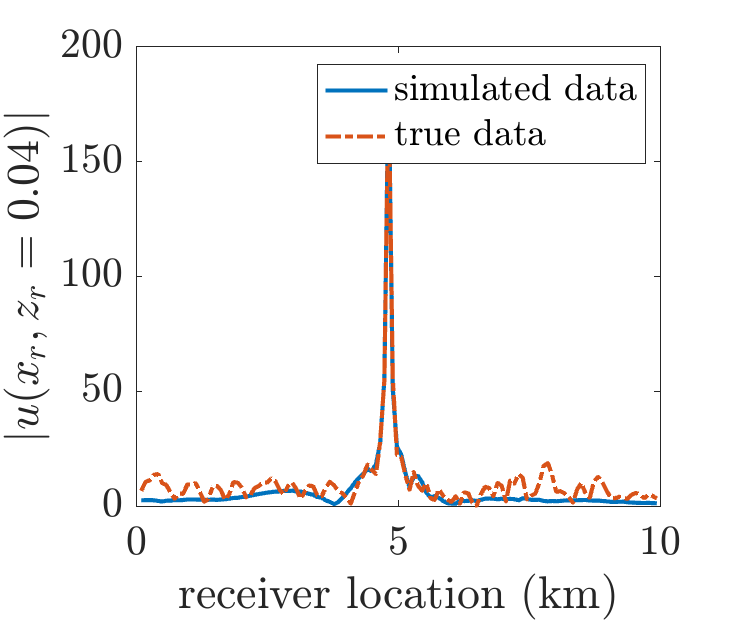}
\includegraphics[width=0.33\linewidth]{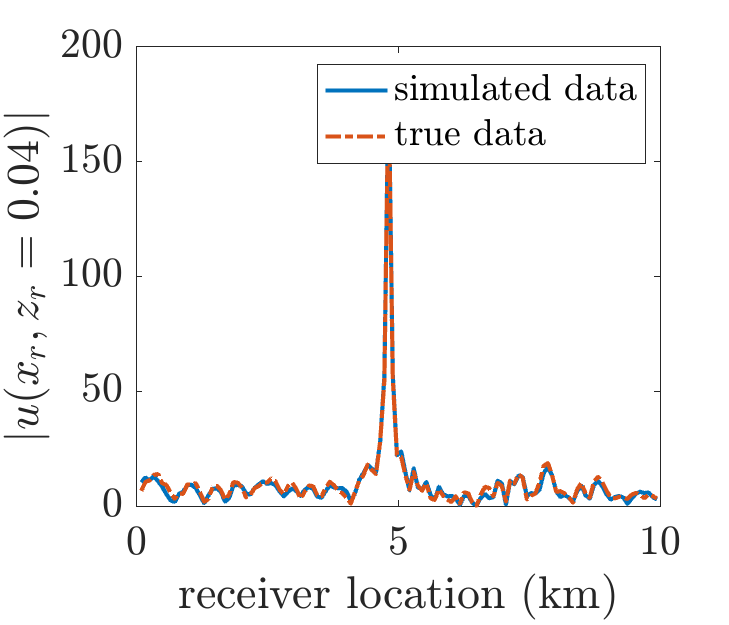}\includegraphics[width=0.33\linewidth]{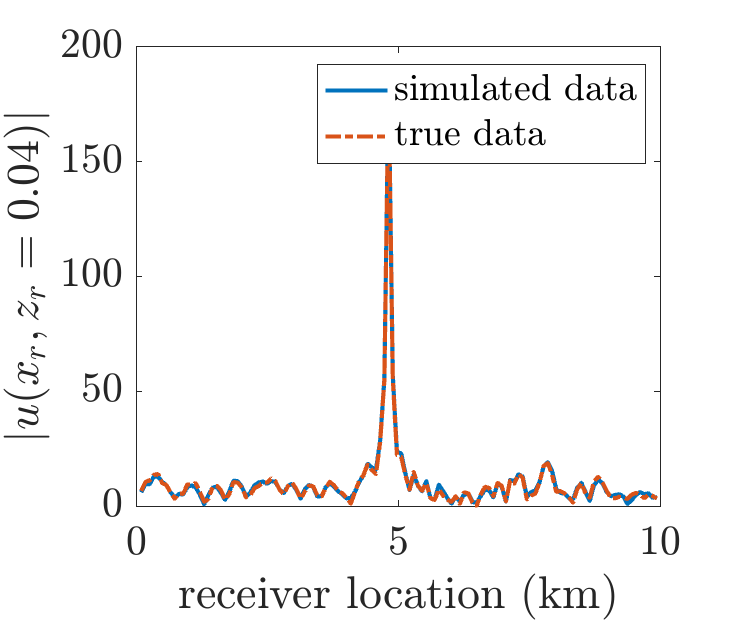}
\caption{The top row shows the reconstructed velocities. The bottom row compares simulated and true data from one source. From left to right shows the result using the objective function~\eqref{conventional}, \eqref{variable metric}, and~\eqref{data-driven metric}, respectively. This sets of result correspond to Case 2, i.e., the initial velocity is the right panel of Figure~\ref{fig:marm-vels}.}
\label{fig:line-vels}
\end{figure}

\section{Conclusion and Discussions}\label{conclusions}
We provided a unified framework for analyzing and discretizing certain PDE-constrained optimization problems arising in inverse problems. Here, the inverse problem is to estimate a spatially varying coefficient (i.e., the parameter) from a finite number of linear measurements of the solution of the PDE (i.e., the state). In particular, we consider constraint relaxation and show that the joint parameter-state estimation problem can be reduced to a variable-metric formulation depending on the parameter alone. This allows us to have two limiting cases, and we interpret these in terms of projections of certain infinite-dimensional residuals onto finite-dimensional subspaces. One limit yields the conventional reduced approach, while the other yields a residual that is affine with respect to the parameter if the PDE operator is affine with respect to the parameter. The latter is arguably more attractive for gradient-based optimization methods. Under further assumptions, we can show that this limiting case, in fact, yields a quadratic problem in terms of the parameter. 

Moreover, the reduced formulation  is the starting point for deriving a data-driven approach, where the variable metric is replaced by one estimated from the measurements. In a few well-chosen case studies, we show how the framework can be applied and how the metric can be computed from the available measurements. In some simple cases, we numerically demonstrate that the data-driven relaxation indeed yields a convex optimization problem with respect to the unknown coefficient. A more complex numerical example in 2D seismic inversion showcases that the variable metric and data-drive metric can achieve a better data-fit and qualitatively superior results to the conventional approach.

The proposed framework forms a natural starting point for implementing the relaxed formulation using the finite element method. For practical implementation on large-scale problems, the matrix defining the metric (either variable or data-driven) will need to be approximated. Whether such approximations will retain the benefits of constraint relaxation will need to be investigated further. Another practical issue we leave for future work is the effect of noise on the data-driven metric. We furthermore expect that stronger statements about the quadratic nature of the relaxed problem can be made in specific cases, and this also is the subject of ongoing research.

\section*{Acknowledgements}
The first author thanks Gabrio Rizzuti, Felix Herrmann, and Bill Symes for the numerous fruitful discussions about this topic. The second author is partially supported by the Office of Naval Research (ONR) under grant N00014-24-1-2088 and the National Science Foundation under grant DMS-2409855.  We also gratefully acknowledge the Banff International Research Station for their support of the workshop \emph{New Ideas in Computational Inverse Problems} (22w5118), during which the foundations of this paper were laid.

\appendix
\section{Auxiliary results}
\begin{lma}\label{lemma:optimality}
Given a the functional $J : \mathsf{U}\rightarrow \mathbb{R}$ defined as
\[
J(q) = \sum_{i=1}^n  {\textstyle\frac{1}{2}}|\langle v_i, q \rangle_{\mathsf{U}} - b_i|^2 +  {\textstyle\frac{\rho}{2}} \|q\|_{\mathsf{U}}^2,
\]
where $\{v_i\}_{i=1}^n$ are linearly independent, it admits minimizers of the form
\[
q = \sum_{i=1}^n \overline{\alpha_i} v_i,
\]
where the coefficients $\boldsymbol{\alpha} \in \mathbb{C}^n$ are solved from
\[
\left(G + \rho I\right)\boldsymbol{\alpha} = \mathbf{b},
\]
with $G_{ij} = \langle v_i, v_j\rangle_\mathsf{U}$.
\end{lma}
\begin{proof}
First, for any $\epsilon>0$ and $h\in U$, consider
\[
J(q + \epsilon h) - J(q) =  {\textstyle\frac{1}{2}}\sum_{i=1}^n \left[ {2} \epsilon \Re  \Big(\left(\langle v_i, q \rangle -  b_i\right)\langle h, v_i\rangle\Big) + \epsilon^2 |\langle h, v_i \rangle|^2\right] +  {\textstyle\frac{\rho}{2}}\left( {2}\epsilon \Re  \langle h, q\rangle + \epsilon^2 \|h\|_\mathsf{U}^2\right)\,,
\]
where ``$\Re\, x $'' denotes the real part of the complex argument $x$.

Note that ignoring the $\mathcal{O}(\epsilon^2)$ terms does not give us the Fr\'echet derivative since the resulting operator is not a linear operator over $\mathbb{C}$. We address this by introducing $q = q_R + \imath q_I$ with $\imath$ being the imaginary unit and $q_R$ and $q_I$ real-valued. Consider the Fr\'echet derivative with respect to $(q_R, q_I)$ evaluated at the direction $(h_R , h_I)$:
\[
DJ_{(q_R,q_I)}(h_R,h_I) = \Re\left[\left\langle h_R, \sum_{i=1}^n \beta_i v_i + \rho q \right\rangle + \imath \left\langle h_I,\sum_{i=1}^n \beta_i v_i + \rho q \right\rangle\right],
\]
with $\beta_i = \overline{\langle v_i, q \rangle -  b_i}$. The optimality condition, $DJ_{(q_R,q_I)} = 0$, is satisfied by letting $q$ be of the form
\[
q = \sum_{i=1}^n \overline{\alpha_i} v_i,
\]
where the coefficients $\boldsymbol{\alpha} = [\alpha_1,\cdots,\alpha_n]^\top \in \mathbb{C}^n$ can be determined by plugging this expression for $q$ back into $J$:
\[
J(\boldsymbol{\alpha}) =  {\textstyle\frac{1}{2}}\|G\boldsymbol{\alpha} - \mathbf{b}\|_2^2 +  {\textstyle\frac{\rho}{2}}\boldsymbol{\alpha}^*G\boldsymbol{\alpha},
\]
with $G_{ij} = \langle v_i, v_j\rangle_\mathsf{U}$.
The corresponding normal equations are
\[
\left(G^2 + \rho G\right)\boldsymbol{\alpha} = G\mathbf{b},
\]
which reduces to
\[
\left(G+ \rho I\right)\boldsymbol{\alpha} = \mathbf{b},
\]
because $G$ has full rank as a result of $\{v_i\}_{i=1}^n$ being linearly independent.
\end{proof}

\begin{lma}
\label{lma:hlmholtz1}
Let 
\[
\mathsf{U} = \left\{ u : [0,1]\rightarrow \mathbb{C} \, \left|\, \|u\|_{\mathsf{U}} < \infty, u(0) = 0\right.\right\},
\]
with $\langle u, v \rangle_\mathsf{U} = \int_0^1 u'(x)\overline{v'(x)}\mathrm{d}x$ and define its dual in the usual way.  Consider the weak formulations for the forward and adjoint 1D Helmholtz equation
\[
\mathcal{A}_{c,k}(u(\cdot\,;\, k), \phi) = \mathcal{P}(\phi)\quad \forall \phi \in \mathsf{U},
\]
\[
\overline{\mathcal{A}_{c,k}(\phi, w(\cdot \,;\, k))} = \mathcal{P}(\phi)\quad \forall \phi \in \mathsf{U},
\]
with
\[
\mathcal{A}_{c,k}(u,v) = \int_0^1 u'(x) \overline{v'(x)}\mathrm{d}x - k^2 \int_0^1 c(x)^{-2}u(x) \overline{v(x)}\mathrm{d}x -\imath k c(1)^{-1}u(1)\overline{v(1)},
\]
\[
\mathcal{P}(\phi) = \int_0^1 f(x)\overline{\phi(x)}\mathrm{d}x,
\]
for real-valued $f\in H^{-1}$. These weak formulations are well-posed \cite{Ihlenburg1997} and we have the following relation between their solutions
\[
w(\cdot,k) = \overline{u(\cdot,k)}.
\]
\end{lma}
\begin{proof}
First note that the adjoint solution satisfies
\[
\mathcal{A}_{c,k}(\phi, w(\cdot \,;\, k)) = \overline{\mathcal{P}(\phi)} = \mathcal{P}(\overline{\phi})\quad \forall \phi \in \mathsf{U},
\]
since $f$ is real-valued. Note furthermore that
\[
\mathcal{A}_{c,k}(u,v) = \mathcal{A}_{c,k}(\overline{v},\overline{u})\quad \forall u,v\in\mathsf{U}.
\]
Thus the solution to the forward problem, $u(\cdot;k)$, satisfies
\[
\mathcal{A}_{c,k}(\overline{\phi},\overline{u(\cdot;k)})=\mathcal{P}({\phi})\quad\forall\phi\in\mathsf{U}.
\]
Introducing $\psi = \overline{\phi}$ yields
\[
\mathcal{A}_{c,k}(\psi,\overline{u(\cdot;k)})=\mathcal{P}(\overline{\psi}) =\overline{\mathcal{P}(\psi)}\quad\forall\psi\in\mathsf{U}.
\]
Thus, $\overline{u(\cdot;k)}$ satisfies the adjoint equation, implying that the solutions to the forward and adjoint equations are related as $w(\cdot;k)=\overline{u(\cdot;k)}$ as stated.
\end{proof}

\begin{lma}
\label{lma:hlmholtz3}
Let $d_{ij} = \mathcal{P}_i(u_j)$ with $u_j$ the solution of $\mathcal{A}_{c,k}(u_j,\phi)=\mathcal{P}_j(\phi)$, $\forall \phi \in \mathsf{U}$, as defined in Lemma \ref{lma:hlmholtz1}. Then we have 
\[
d_{ij} = d_{ji}.
\]
\end{lma}
\begin{proof}
We have $d_{ij}=\mathcal{P}_i(u_j) = \overline{\mathcal{A}_{c,k}(u_j,w_i)} = \overline{\mathcal{P}_j(w_i)}$ using the definition of the adjoint equation. 
Using the fact that $f_i$ is real-valued we have $\overline{\mathcal{P}_j(w_i)}=\mathcal{P}_j(\overline{w_i})$. Using Lemma~\ref{lma:hlmholtz1}, this yields $\mathcal{P}_j(\overline{w_i}) = \mathcal{P}_j(u_i) = d_{ji}$, completing the proof.
\end{proof}

\begin{lma}
\label{lma:hlmholtz2}
Define
\[
g_{ij}(k) = \langle w_i(\cdot; k),  w_j(\cdot; k)\rangle_{\mathsf{U}},
\]
where $w_i(\cdot; k)$ is a (weak) solution of the adjoint 1D Helmholtz equation, as defined in Lemma~\ref{lma:hlmholtz1}. 
Furthermore, let $d_{ij}(k) = \mathcal{P}_i(u_j(\cdot;k))$ and $b_i(k) = u_i(1;k)$ with $u_i(\cdot; k)$ a (weak) solution of the 1D Helmholtz equation. Then we have
\begin{equation}
g_{ij}(k) = \Re\left(d_{ij}(k) + {\textstyle\frac{k}{2}}d_{ij}'(k)\right) + {\textstyle\frac{\imath k^2}{2c(1)}}\left(\overline{b_i'(k)}b_j(k) - \overline{b_i(k)}b_j'(k)\right),
\end{equation}
where $d_{ij}'$ and $b_i'$ denote the derivatives of $d_{ij}$ and $b_i$ with respect to $k$. 
\end{lma}
\begin{proof}
First note that we can express
\[
g_{ij}(k) = \langle u_j(\cdot; k),  u_i(\cdot; k)\rangle_{\mathsf{U}},
\]
because $\overline{u_i} = w_i$ (see Lemma \ref{lma:hlmholtz1}). Introducing the sesquilinear forms $\mathcal{M}(u,v) = \int_0^1 c(x)^{-2}u(x)\overline{v(x)}\mathrm{d}x$ and $\mathcal{B}(u,v) = c(1)^{-1}u(1)\overline{v(1)}$ and use $u_j(\cdot,\ell)$ as test function in the weak form, we find
\[
d_{ij}(\ell) = \langle u_i(\cdot;k),  u_j(\cdot;\ell)\rangle_{\mathsf{U}} - k^2 \mathcal{M}(u_i(\cdot;k),  u_j(\cdot;\ell)) - \imath k \mathcal{B}( u_i(\cdot;k),  u_j(\cdot;\ell) ).
\]
Similarly,
\[
\overline{d_{ji}(k)} = \langle u_i(\cdot;k),  u_j(\cdot;\ell)\rangle_{\mathsf{U}} - \ell^2 \mathcal{M}(u_i(\cdot;k),  u_j(\cdot;\ell)) + \imath \ell \mathcal{B}( u_i(\cdot;k),  u_j(\cdot;\ell) ).
\]
Combining the relations for $\ell^2 d_{ij}(\ell)$ and $k^2 \overline{d_{ji}(k)}$ and using that $d_{ij} = d_{ji}$ (see~Lemma \ref{lma:hlmholtz3}) we find
\[
\langle u_i(\cdot;k), u_j(\cdot;\ell)\rangle_\mathsf{U} = \frac{\ell^2 d_{ij}(\ell) - k^2 \overline{d_{ij}(k)}}{(\ell+k)(\ell-k)} + \frac{\imath k\ell}{\ell-k} \mathcal{B}(u_i(\cdot;k),u_j(\cdot;\ell)).
\]
On the other hand, we can express $g_{ij}(k)$ as
\[
g_{ij}(k) = \lim_{\ell\rightarrow k} {\textstyle\frac{1}{2}}\left(\langle u_j(\cdot;k), u_i(\cdot;\ell)\rangle_\mathsf{U} + \langle u_j(\cdot;\ell), u_i(\cdot;k)\rangle_\mathsf{U}\right).
\]
This yields
\[
g_{ij}(k) = \lim_{\ell\rightarrow k} \frac{\ell^2 \Re d_{ij}(\ell) - k^2 \Re{d_{ij}(k)}}{(\ell+k)(\ell-k)} + \frac{\imath k\ell}{2(\ell - k)}\left(\mathcal{B}(u_j\cdot;k),u_i(\cdot;\ell)) - \mathcal{B}(u_j(\cdot;\ell),u_i(\cdot;k))\right).
\]
We readily observe that the first term yields $\Re (d_{ij}(k) + \frac{k}{2}d'_{ij}(k))$. For the second term, observe that we can express it as
\[
\frac{\imath k\ell}{2(\ell - k)} \left(\mathcal{B}(u_j(\cdot;k) - u_j(\cdot;\ell),u_i(\cdot;\ell)) + \mathcal{B}(u_j(\cdot;\ell),u_i(\cdot;\ell)-u_i(\cdot;k))\right).
\]
This yields
\[
g_{ij}(k) = \Re \left(d_{ij}(k) + \frac{k}{2}d'_{ij}(k)\right) + \frac{\imath k^2}{2}\left(\mathcal{B}\left(u_j(\cdot;k),\frac{\mathrm{d} u_i}{\mathrm{d}k}(\cdot;k)\right)-\mathcal{B}\left(\frac{\mathrm{d} u_j}{\mathrm{d}k}(\cdot;k),u_i(\cdot;k)\right)\right).
\]
Introducing $b_i(k) = u_i(1;k)$ and using the definition of $\mathcal{B}$ yields the desired result.
\end{proof}

\end{document}